\documentclass[11pt]{amsart}
\usepackage{amssymb}
\usepackage{amsmath}
\usepackage{amsfonts}
\usepackage{graphicx}

\usepackage{hyperref}
    \usepackage{aeguill}
    \usepackage{type1cm}

\theoremstyle{plain}

\newtheorem{claim}{Claim}

\newtheorem{corollary}{Corollary}

\newtheorem{definition}{Definition}
\newtheorem{example}{Example}

\newtheorem{lemma}{Lemma}

\newtheorem{proposition}{Proposition}
\newtheorem{remark}{Remark}

\newtheorem{theorem}{Theorem}
\newcommand{\N}{\mathbb{N}}

\newcommand{\R}{\mathbb{R}}

\numberwithin{equation}{section}

\begin{document}

\title[Global and fine approximation of convex functions]{Global and fine approximation of convex functions}

\author{Daniel Azagra}

\address{ICMAT (CSIC-UAM-UC3-UCM), Departamento de An{\'a}lisis Matem{\'a}tico,
Facultad Ciencias Matem{\'a}ticas, Universidad Complutense, 28040, Madrid, Spain}
\email{daniel\_azagra@mat.ucm.es}

\date{January 27, 2012}

\dedicatory{Dedicated to the memory of Robb Fry}

\maketitle

\begin{abstract}
Let $U\subseteq\R^d$ be open and convex. We prove that every
(not necessarily Lipschitz or strongly) convex function $f:U\to\R$ can be approximated by real
analytic convex functions, uniformly on all of $U$.
We also show that $C^0$-fine approximation of convex functions by smooth
(or real analytic) convex functions on $\R^d$ is possible in general
if and only if $d=1$. Nevertheless, for $d\geq 2$ we give a
characterization of the class of convex functions on $\R^d$ which
can be approximated by real analytic (or just smoother) convex
functions in the $C^0$-fine topology. It turns out that the
possibility of performing this kind of approximation is not
determined by the degree of local convexity or smoothness of the
given function, but by its global geometrical behaviour. We
also show that every $C^{1}$ convex and proper function on
$U$ can be approximated by
$C^{\infty}$ convex functions in the $C^{1}$-fine topology, and we provide
some applications of these results, concerning prescription of (sub-)differential
boundary data to convex real analytic functions, and smooth surgery
of convex bodies.
\end{abstract}

\section{Introduction and main results}

Two important classes of functions in analysis and in geometry are those of
Lipschitz functions and convex functions $f:U\subseteq \R^d\to\R$.
Although these functions are almost everywhere differentiable (or
even almost everywhere twice differentiable in the convex case), it is sometimes useful
to approximate them by smooth functions which are
Lipschitz or convex as well.

In the case of a Lipschitz function $f:U\subseteq\R^d\to\R$, this can easily
be done as follows: by considering the function
$x\mapsto \inf_{y\in U}\{f(y)+L|x-y|\}$ (where
$L=\textrm{Lip}(f)$, the Lipschitz constant of $f$), which is a
Lipschitz extension of $f$ to all of $\R^d$ having the same
Lipschitz constant, one can assume $U=\R^d$. Then, by setting
$f_{\varepsilon}=f*H_{\varepsilon}$, where
$H_{\varepsilon}(x)=\frac{1}{(4\pi
    \varepsilon)^{d/2}}\exp(-|x|^2/4\varepsilon)$ is the heat kernel,
    one obtains real analytic
Lipschitz functions (with the same Lipschitz constants as $f$)
which converge to $f$ uniformly on all of $\R^d$ as
$\varepsilon\searrow 0$. If one replaces $H_\varepsilon$ with any
approximate identity $\{\delta_\varepsilon\}_{\varepsilon>0}$ of
class $C^k$, one obtains $C^k$ Lipschitz approximations. Moreover,
if $\delta_{\varepsilon}\geq 0$ and $f$ is convex, then
the functions $f_\varepsilon$ are convex as well.

However, if $f:\R^d\to\R$ is convex but not globally Lipschitz, the
convolutions $f*H_{\varepsilon}$ may not be well defined or, even
when they are well defined, they do not converge to $f$ uniformly
on $\R^d$. On the other hand, the convolutions
$f*\delta_\varepsilon$ (where
$\delta_{\varepsilon}=\varepsilon^{-d}\delta(x/\varepsilon)$,
$\delta\geq 0$ being a $C^\infty$ function with bounded support and
$\int_{\R^d}\delta=1$) are always well defined, but they only
provide uniform approximation of $f$ on {\em compact} sets. Now,
partitions of unity cannot be used to glue
these local convex approximations into a global approximation,
because they do not preserve convexity. To see why this is so, let
us consider the simple case of a $C^2$ convex function
$f:\R\to\R$, to be approximated by $C^\infty$ convex functions.
Take two bounded intervals $I_1\subset I_2$, and $C^\infty$
functions $\theta_1, \theta_2:\R\to [0,1]$ such that
$\theta_1+\theta_2=1$ on $\R$, $\theta_1 =1$ on $I_1$, and
$\theta_2=1$ on $\R\setminus I_2$. Given $\varepsilon_j>0$ one may
find $C^\infty$ convex functions $g_j$ such that $\max\{|f-g_j|,
|f'-g'_j|, |f''-g''_j|\}\leq\varepsilon_j$ on $I_j$. If
$g=\theta_1 g_1+\theta_2 g_2$ one has
    $$
g''=g''_1\theta_1+g''_2\theta_2+ 2(g'_1-g'_2)\theta'_1
+(g_1-g_2)\theta''_1.
    $$
If $f''>0$ on $I_2$ then by choosing $\varepsilon_i$ small enough
one can control this sum and get $g''\geq 0$, but if the $g''_i=0$
vanish somewhere there is no way to do this (even if we managed to
have $g_2\geq g_1$ and $g'_2\geq g'_1$, as $\theta''_1$ must
change signs).

In \cite{Greene5}, \cite{Greene3}, \cite{Greene4} Greene and Wu studied the
question of approximating a convex function defined on a
(finite-dimensional) Riemannian manifold $M$\footnote{In
Riemannian geometry convex functions have been used, for instance,
in the investigation of the structure of noncompact manifolds of
positive curvature by Cheeger, Greene, Gromoll, Meyer, Siohama, Wu
and others, see \cite{GromollMeyer}, \cite{CheegerGromoll}, \cite{Greene1},
\cite{Greene2}, \cite{Greene3}, \cite{Greene4}. The existence of global convex
functions on a Riemannian manifold has strong geometrical and
topological implications. For instance \cite{Greene1}, every
two-dimensional manifold which admits a global convex function
that is locally nonconstant must be diffeomorphic to the plane,
the cylinder, or the open M{\"o}bius strip.}, and they showed that if
$f:M\to\R$ is strongly convex (in the sense of the following
definition), then for every $\varepsilon>0$ one can find a
$C^\infty$ strongly convex function $g$ such that
$|f-g|\leq\varepsilon$ on all of $M$.

\begin{definition}
A $C^{2}$ function $\varphi:M\to\mathbb{R}$
is called strongly convex if its second derivative along any nonconstant
geodesic is strictly positive everywhere on the geodesic. A (not
necessarily smooth) function $f:M\to\mathbb{R}$ is said to be
strongly convex provided that for every $p\in M$ there exists an open neighbourhood $V$ of $p$ and a
strongly convex function $\varphi\in C^{2}(V)$ such that
$f-\varphi$ is convex on $V$.\footnote{We
warn the reader that, in Greene and Wu's papers, what we have just
defined as strong convexity is called strict convexity. We have
changed their terminology since we will be mainly concerned with
the case $M=\R^d$, where one traditionally defines a strictly
convex function as a function $f$ satisfying $f\left(
(1-t)x+ty\right) < (1-t)f(x)+tf(y)$ if $0<t<1$.}
\end{definition}

This solves the problem when the given function $f$ is strongly
convex. However, as Greene and Wu pointed out, their method cannot
be used when $f$ is not strongly convex. This is inconvenient
because strong convexity is a very strong condition: for instance,
the function $f(x)=x^4$ is strictly convex, but not strongly
convex on any neighbourhood of $0$. However, as shown by Smith in
\cite{Smith}, this is a necessary condition in the general
Riemannian setting: for each $k=0,1,..., \infty$, there exists a
flat Riemannian manifold $M$ such that on $M$ there is a $C^k$
convex function which cannot be globally approximated by a
$C^{k+1}$ convex function (here $C^{\infty+1}$ means real
analytic). There are no results characterizing the manifolds on
which global approximation of convex functions by smooth convex
functions is possible. Even in the most basic case $M=\R^d$,
we have been unable to find any reference dealing with the problem of finding smooth global
approximations of (not necessarily Lipschitz or strongly) convex functions.

One of the main purposes of this paper is proving the following.

\begin{theorem}\label{uniform approximation of convex by real analytic convex}
Let $U\subseteq\R^d$ be open and convex. For every convex
function $f: U\to\R$ and every $\varepsilon>0$ there exists a
real-analytic convex function $g:U\to\R$ such that $f-\varepsilon
\leq g\leq f$.
\end{theorem}

This result is optimal in several ways: as we will see, it is not possible to obtain $C^{0}$-fine approximation of convex functions by $C^1$ convex functions on $\R^d$ when $d\geq 2$ (and even in the case $d=1$ this kind of approximation is not possible from below).

In showing this theorem we will develop a gluing technique for
convex functions which will prove to be useful also in the setting
of Riemannian manifolds or Banach spaces.

\begin{definition}\label{defn approx from below}
Let $X$ be $\R^d$, or a complete Riemannian manifold (not necessarily
finite-dimensional), or a Banach space, and let $U\subseteq X$ be open and convex.
We will say that a continuous convex function $f:U\to\R$ can
be approximated from below by $C^k$ convex functions, uniformly on
bounded subsets of $U$, provided that for every bounded set
$B$ with $\overline{B}\subset U$ and $\textrm{dist}(B, \partial U)>0$, and for every $\varepsilon>0$ there
exists a $C^k$ convex function $g:U\to\R$ such that
\begin{enumerate}
\item $g\leq f$ on $U$, and
\item $f-\varepsilon\leq g$ on $B$.
\end{enumerate}
\end{definition}
(In the case when $U=X$ is unbounded we will use the convention that $\textrm{dist}(B, \partial U)=\infty$ for every bounded set $B\subset X$.)

\begin{theorem}[Gluing convex approximations]\label{main theorem}
Let $X$ be $\R^d$, or a complete Riemannian manifold (not necessarily
finite-dimensional), or a Banach space, and let $U\subseteq X$ be open and convex.
Assume that $U=\bigcup_{n=1}^{\infty}B_n$, where the $B_n$ are open bounded
convex sets such that $\textrm{dist}(B_n, \partial U)>0$ and $\overline{B_n}\subset B_{n+1}$
for each $n$. Assume also that $U$ has the property that every continuous, convex function $f:U\to\R$ can be
approximated from below by $C^k$ convex (resp. strongly convex) functions ($k\in\N\cup\{\infty\}$),
uniformly on bounded subsets of $U$.

Then every continuous convex function $f:U\to\R$ can be
approximated from below by $C^k$ convex (resp. strongly convex) functions, uniformly on
$U$.
\end{theorem}

From this result (and from its proof and the known results on approximation on bounded sets) we will easily deduce the
following corollaries.

\begin{corollary}\label{uniform approximation of convex functions by smooth convex functions in Rn}
Let $U\subseteq\R^d$ be open and convex. For every convex function
$f: U\to\R$ and every $\varepsilon>0$ there exists a $C^\infty$
convex function $g:U\to\R$ such that $f-\varepsilon \leq g\leq f$.
Moreover $g$ can be taken so as to preserve local Lipschitz
constants of $f$ (meaning $\textrm{Lip}(g_{|_B})\leq
\textrm{Lip}(f_{|_{(1+\varepsilon)B}})$ for every ball $B\subset
U$). And if $f$ is strictly (or strongly) convex, so can $g$ be
chosen.
\end{corollary}

\begin{corollary}\label{uniform approximation of convex functions on Cartan Hadamard manifolds}
Let $M$ be a Cartan-Hadamard Riemannian
manifold (not necessarily finite dimensional), and $U\subseteq M$ be open and convex. For every
convex function $f: U\to\R$ which is bounded on bounded subsets $B$ of $U$ with $\textrm{dist}(B, \partial U)>0$,
and for every $\varepsilon>0$ there exists
a $C^1$ convex function $g:U\to\R$ such that $f-\varepsilon \leq
g\leq f$. Moreover $g$ can be chosen so as to preserve the set of
minimizers and the local Lipschitz constants of $f$. And, if $f$ is strictly convex and $M$ is finite dimensional, $g$ can be
taken to be strictly convex as well.
\end{corollary}

One should expect that the above corollary is not optimal
(in that approximation by $C^{\infty}$ convex functions should be possible).

\begin{corollary}\label{uniform approximation of convex functions on Banach spaces}
Let $X$ be a Banach space whose dual is locally uniformly convex, and $U\subseteq X$ be open and convex.
For every convex function $f: U\to\R$ which is bounded on bounded subsets $B$ of $U$ with
$\textrm{dist}(B, \partial U)>0$, and for every
$\varepsilon>0$ there exists a $C^1$ convex function $g:U\to\R$
such that $f-\varepsilon \leq g\leq f$. Moreover $g$ can be taken
so as to preserve the set of minimizers and the local Lipschitz
constants of $f$. And if $f$ is strictly convex and $X$ is reflexive, $g$ can be taken to be strictly convex as well.
\end{corollary}

A question remains open whether every convex
function $f$ defined on a separable infinite-dimensional Hilbert
space $X$ which is bounded on bounded sets can be globally approximated by $C^2$ convex functions
(notice that Theorem \ref{main theorem} cannot be combined with
the results of \cite{DFH1}, \cite{DFH2} on smooth and real analytic
approximation of bounded convex bodies in Banach spaces in order to give a solution
to this problem. Although one can use these results, together with
the implicit function theorem, to find smooth convex approximations of
$f$ on a bounded set, the approximating functions obtained by this
process are not defined on all of $X$ and are not strongly convex,
hence it is not clear how to extend them to a smooth convex function below $f$ on $X$,
or even if this should be possible at all).

As a byproduct of the proof of Theorem \ref{uniform
approximation of convex by real analytic convex} we will also
obtain the following characterization of the class of convex
functions that can be globally approximated by strongly convex
functions on $\R^d$.

\begin{proposition}\label{characterization of functions that cannot be approximated by strongly convex functions}
Let $f:\R^d\to\R$ be a convex function. The following
conditions are equivalent:
\begin{enumerate}
\item $f$ cannot be uniformly approximated by strictly convex
functions.
\item $f$ cannot be uniformly approximated by strongly convex
functions.
\item There exist $k<d$, a linear projection $P:\R^d\to\R^k$, a
convex function $c:\R^k\to\R$ and a linear function $\ell:\R^d\to\R$
such that $f=c\circ P +\ell$.
\end{enumerate}
\end{proposition}

We will also consider fine approximation of convex functions on subsets of $\R^d$.
In this direction, the only known results concerning $C^0$-fine
approximation of convex functions by smooth convex functions are also due
to Greene and Wu \cite{Greene4}, who showed that every {\em strongly
convex} function $f$ defined on a (finite-dimensional)
Riemannian manifold $M$ can be approximated by $C^{\infty}$ strongly
convex functions in the $C^{0}$-fine topology.

We say that a convex function $f\in C^k(M)$ can be approximated by $C^{\infty}$ convex functions in the $C^k$-fine topology provided that
for every continuous function $\varepsilon:M\to (0, \infty)$ there exists a
convex function $g\in C^{\infty}(M)$ such that $|f-g|\leq\varepsilon$ and
$\|D^{j}f-D^{j}g\|\leq\varepsilon$ on $M$ for $j\leq k$ when $k\geq
1$.

For $d=1$ we have the following.
\begin{theorem}\label{C0 fine approximation of convex functions on R}
Let $U\subseteq\R$ be an open interval.
Every convex function $f:U\to\R$ can be approximated by real analytic convex functions in the $C^0$-fine topology.
\end{theorem}

For $d\geq 2$, we will provide a characterization of
the class of convex functions on $\R^d$ which can be approximated in
the $C^0$-fine topology by smoother (or real analytic) convex
functions. Interestingly, the possibility of performing this kind of
approximation has very little to do with the degree of local
convexity or smoothness of the given function. It is the global
geometrical behaviour of the function that determines whether or not
it can be approximated by more regular convex functions in this
topology.

\begin{definition}
{\em Let $U\subseteq\R^d$ be open and convex. We will say that a
function $f:U\to\R$ is {\em properly convex} provided that $f=\ell +c$, where
$\ell$ is linear, $c:U\mapsto [a, b)$, $-\infty< a<b\leq \infty$, and $c$
is convex and proper (meaning that $c^{-1}[a, \beta]$ is compact
for every $\beta\in [a, b)$).}
\end{definition}

It is obvious that proper convexity  is not a
restrictive property from a local point of view, but it has
global geometrical implications.

\begin{theorem}\label{characterization of the fine approximation property for convex functions on Rd}
Let $f:\R^d\to\R$ be a $C^p$ convex function which is not of class $C^{p+1}$, $p\in\N\cup\{\infty\}$, $d\geq 2$. The following statements are equivalent:
\begin{enumerate}
\item $f$ is properly convex.
\item $f$ can be written in the form $f=\ell +c$, where $\ell$ is linear and $\lim_{|x|\to\infty} c(x)=\infty$.
\item $f$ cannot be written in the form $f=c\circ P +\ell$, where $P:\R^d\to\R^k$ is a linear projection, $k<d$, $c:\R^k\to\R$ is convex, and $\ell:\R^d\to\R$ is linear.
\item $f$ can be approximated by strongly convex real analytic functions in the $C^0$-fine topology.
\item $f$ can be approximated by $C^{p+1}$ convex functions in the $C^0$-fine topology.
\end{enumerate}
\end{theorem}

In the case when $U$ is a proper open convex subset of $\R^d$, $d\geq 2$, it
would be harder to establish a full characterization (in the spirit
of the preceding theorem) of the class of convex functions
$f:U\to\R$ which can be approximated by smoother convex functions.
We do not embark on such a program, but we do prove that every
properly convex function on $U$ can be approximated by convex real
analytic functions in the $C^0$-fine topology.

\begin{theorem}\label{a sufficent condition for fine approximation property for convex functions on Rd}
Let $U\subseteq\R^d$ be open and convex, and $f:U\to\R$ be properly convex. Then $f$ can be
approximated by strongly convex real analytic functions in the
$C^0$-fine topology.
\end{theorem}

When $f\in C^{1}$, we will show a slightly weaker result (but still powerful enough to imply quite interesting geometrical corollaries):
we are able to approximate any $C^{1}$ properly convex function by
$C^{\infty}$ convex functions in the $C^{1}$-fine topology.

\begin{theorem}\label{a sufficent condition for C1 fine approximation property for convex functions on Rd}
Let $U\subseteq\R^d$ be open and convex, and $f:U\to\R$ be properly
convex and $C^1$. Then $f$ can be approximated by $C^{\infty}$ convex
functions in the $C^1$-fine topology.
\end{theorem}

We will also show (see Example \ref{counterexample to properly
convex approximation} below) that on $(-1,1)\times
(-1,1)\subset\R^2$ there exists a $C^p$, but not $C^{p+1}$, convex
function $f$ which is affine exactly on a very thin neighbourhood of a line, which is strongly
convex outside a very small neighbourhood of this line, and which cannot
be approximated by $C^{p+1}$ convex functions in the $C^0$-fine
topology. Hence, even in the case $U\neq\R^d$, proper convexity
is a very reasonable condition to require of a nonsmooth convex function, if
one wants to approximate it by smooth convex functions.

As a first geometrical application of Theorem \ref{a sufficent condition
for fine approximation property for convex functions on Rd}, we will
show that one can sometimes prescribe subdifferential data to real
analytic convex functions at the boundary of a compact convex body.

\begin{corollary}\label{prescribing continuous boundary values and subdifferentials of real analytic convex functions}
Let $U\subseteq\R^d$ be open and convex, $f:U\to\R$ be a convex function of the form $f=\ell +c$, where $\ell$ is linear and $c$ is proper, $K$ a compact convex body of
the form $K=c^{-1}(-\infty, b]$, and $\varepsilon:\rm{int}(K)\to (0,
\infty)$ a continuous function. Then there exists a convex function
$F:U\to\R$ such that
\begin{enumerate}
\item $F=f$ on $U\setminus \textrm{int}(K)$
\item $|F-f|\leq\varepsilon$ on $\textrm{int}(K)$
\item $F$ is strongly convex and real analytic on $\textrm{int}(K)$
\item $\partial F(x)=\partial f(x)$ for each $x\in\partial K$.
\end{enumerate}
Moreover, if $f\in C^{1}(U\setminus\textrm{int}(K))$ then $F\in C^{1}(U)$.
\end{corollary}
\noindent As is usual, we denote $\partial F(x)=\{\zeta:\R^{n}\to\R
\,\, | \,\, \zeta \textrm{ is linear }, F(y)-F(x)\geq \zeta(x-y)
\textrm{ for all } y\in U\}$, the subdifferential of $F$.

In the case when the given function $f$ is already $C^2$ outside $\textrm{int}(K)$, we will also
show the following.

\begin{corollary}\label{prescribing differentials of smooth convex functions}
Let $U\subseteq\R^d$ be open and convex, $f:U\to\R$ be a convex function of the form $f=\ell +c$, where $\ell$ is linear and $c$ is proper, $K$ a compact convex body of the form $K=c^{-1}(-\infty,
b]$, and $\varepsilon:\rm{int}(K)\to (0, \infty)$ a continuous
function. Assume that $f$ is $C^{2}$ on $U\setminus\textrm{int}(K)$.
Then there exists a $C^2$ convex function $F:U\to\R$ such that
\begin{enumerate}
\item $F=f$ on $U\setminus \textrm{int}(K)$
\item $|F-f|\leq\varepsilon$ on $\textrm{int}(K)$
\item $F$ is $C^{\infty}$ on $\textrm{int}(K)$.
\end{enumerate}
\end{corollary}

These corollaries are in the spirit of Ghomi's work on optimal smoothing of convex functions \cite{Ghomi1}, but note
that here we do not require strong convexity of $f$ on any neighbourhood of $\partial K$.

The above corollaries are also useful in smooth surgery of convex
bodies, e.g. as in the following situation: one has a convex body
with a relatively small part that one does not like (for instance
because it is not sufficiently smooth or convex). Assuming that the
part is a intersection of the given body with a half-space, then one
can replace that part with another piece which approximates the
given part, and which has a smooth boundary, with no loss of first or second order differential
information at the seam.
\begin{corollary}
Let $C$ be a compact convex body in $\R^d$, and let $K$ be a convex
body of the form $K=\ell^{-1}(-\infty,b]\cap C$, where $\ell$ is a
linear function on $\R^d$. Let $P$ be the orthogonal projection of
$\R^d$ onto the subspace $\textrm{Ker} \, \ell$, and assume that
$P(K)$ is contained in the interior of $P(C)$, and that $\partial
C\setminus\textrm{int}(K)$ is contained in a $C^{p}$ convex hypersurface, $p=1,2$. Then,
for every number $\varepsilon>0$ there exists a compact convex body
$D$ such that:
\begin{enumerate}
\item $\partial D$ is a compact $C^p$ convex hypersurface;
\item $C\setminus K=D\setminus K$;
\item $\partial D\cap\ell^{-1}(-\infty, b)$ is a $C^{\infty}$
convex hypersurface (or even real analytic and strongly convex in the case $p=1$);
\item $\rm{dist}\left(\partial C\cap\ell^{-1}(-\infty,b], \, \partial
D\cap \ell^{-1}(-\infty,b]\right)\leq\varepsilon$.
\end{enumerate}
\end{corollary}
One might like to compare the above corollary with the main result of
\cite{Ghomi2}, which provides a procedure for smoothing the edges and
vertices of a convex polytope.

\section{A general gluing technique}

In order to prove Theorem \ref{main theorem} we will use the following.

\begin{lemma}[Smooth maxima]\label{smooth maxima}
For every $\varepsilon>0$ there exists a $C^\infty$ function $M_{\varepsilon}:\R^{2}\to\R$ with the following properties:
\begin{enumerate}
\item $M_{\varepsilon}$ is convex;
\item $\max\{x, y\}\leq M_{\varepsilon}(x, y)\leq \max\{x,y\}+\frac{\varepsilon}{2}$ for all $(x,y)\in\R^2$.
\item $M_{\varepsilon}(x,y)=\max\{x,y\}$ whenever $|x-y|\geq\varepsilon$.
\item $M_{\varepsilon}(x,y)=M_{\varepsilon}(y,x)$.
\item $\textrm{Lip}(M_{\varepsilon})=1$ with respect to the norm $\|\cdot\|_{\infty}$ in $\R^2$.
\item $y-\varepsilon\leq x<x'\implies M_{\varepsilon}(x,y)<M_{\varepsilon}(x',y)$.
\item $x-\varepsilon\leq y<y'\implies M_{\varepsilon}(x,y)<M_{\varepsilon}(x,y')$.
\item $x\leq x', y\leq y' \implies M_{\varepsilon}(x,y)\leq M_{\varepsilon}(x', y')$, with a strict inequality in the case when both $x<x'$ and $y<y'$.
\end{enumerate}
\end{lemma}
We will call $M_{\varepsilon}$ a smooth maximum.
\begin{proof}
It is easy to construct a $C^{\infty}$ function $\theta:\R\to (0,
\infty)$ such that:
\begin{enumerate}
\item $\theta(t)=|t|$ if and only if $|t|\geq\varepsilon$;
\item $\theta$ is convex and symmetric;
\item $\textrm{Lip}(\theta)=1$.
\end{enumerate}
Then it is also easy to check that the function $M_{\varepsilon}$ defined by
$$
M_{\varepsilon}(x,y)=\frac{x+y+\theta(x-y)}{2}
$$
satisfies the required properties. For instance, let us check
properties (v), (vi), (vii) and (viii), which are perhaps less
obvious than the others. Since $\theta$ is $1$-Lipschitz we have
\begin{eqnarray*}
    & & M_{\varepsilon}(x,y)-M_{\varepsilon}(x', y')=\frac{x-x' +y-y'+\theta(x-y)-\theta(x'-y')}{2}\leq \\
    & &\frac{(x-x)' +(y-y')+|x-x'-y+y'|}{2}=\\
    & &\max\{x-x', y-y'\}\leq \max\{|x-x'|, |y-y'|\},
\end{eqnarray*}
which establishes (v). To verify (vi) and (vii), note that our function $\theta$ must satisfy $|\theta'(t)|<1 \iff  |t|<\varepsilon$. Then we have
$$
\frac{\partial M_{\varepsilon}}{\partial x }(x,y)=\frac{1}{2}\left( 1+\theta'(x-y)\right)\geq \frac{1}{2}\left( 1-|\theta'(x-y)|\right)>0 \textrm{ whenever } |x-y|<\varepsilon,
$$
while
$$
\frac{\partial M_{\varepsilon}}{\partial x }(x,y)=\frac{1}{2}\left( 1+\theta'(x-y)\right)= \left\{
                                                    \begin{array}{ll}
                                                      1, & \textrm{ if } x\geq y+\varepsilon, \\
                                                      0, & \textrm{ if } y\geq x+\varepsilon.
                                                    \end{array}
                                                  \right.
$$
This implies (vi) and, together with (iv), also (vii) and the first part of (viii). Finally, if for instance we have $x'>x=\max\{x,y\}$ then $M_{\varepsilon}(x,y)<M_{\varepsilon}(x',y)$ by (vi), and if in addition $y'>y$ then $M_{\varepsilon}(x',y)\leq M_{\varepsilon}(x',y')$ by the first part of (viii), hence $M_{\varepsilon}(x,y)<M_{\varepsilon}(x',y')$. This shows the second part of (viii).
\end{proof}

The smooth maxima $M_{\varepsilon}$ are useful to approximate the
maximum of two functions without losing convexity or other key
properties of the functions, as we next see.

\begin{proposition}\label{properties of M(f,g)}
Let $U\subseteq X$ be as in the statement of Theorem \ref{main theorem},
$M_{\varepsilon}$ as in the preceding Lemma, and let $f, g: U\to\R$
be convex functions. For every $\varepsilon>0$, the function
$M_{\varepsilon}(f,g):U\to\R$ has the following properties:
\begin{enumerate}
\item $M_{\varepsilon}(f,g)$ is convex.
\item If $f$ is $C^k$ on $\{x: f(x)\geq g(x)-\varepsilon\}$ and $g$ is $C^k$ on $\{x: g(x)\geq f(x)-\varepsilon\}$ then $M_{\varepsilon}(f,g)$ is $C^k$ on $U$. In particular, if $f, g$ are $C^k$, then so is $M_{\varepsilon}(f,g)$.
\item $M_{\varepsilon}(f,g)=f$ if $f\geq g+\varepsilon$.
\item $M_{\varepsilon}(f,g)=g$ if $g\geq f+\varepsilon$.
\item $\max\{f,g\}\leq M_{\varepsilon}(f,g)\leq \max\{f,g\} + \varepsilon/2$.
\item $M_{\varepsilon}(f,g)=M_{\varepsilon}(g, f)$.
\item $\textrm{Lip}(M_{\varepsilon}(f,g)_{|_B})\leq \max\{ \textrm{Lip}(f_{|_B}), \textrm{Lip}(g_{|_B}) \}$ for every ball $B\subset U$ (in particular $M_{\varepsilon}(f,g)$ preserves common local Lipschitz constants of $f$ and $g$).
\item If $f, g$ are strictly convex on a set $B\subseteq U$, then so is $M_{\varepsilon}(f,g)$.
\item If $f, g\in C^2(X)$ are strongly convex on a set $B\subseteq U$, then so is $M_{\varepsilon}(f,g)$.
\item If $f_1\leq f_2$ and $g_1\leq g_2$ then $M_{\varepsilon}(f_1, g_1)\leq M_{\varepsilon}(f_2, g_2)$.
\end{enumerate}
\end{proposition}
\begin{proof}
Properties (ii), (iii), (iv), (v), (vi), (vii) and (x) are obvious
from the preceding lemma. To check (i) and (viii), we simply use
(x) and convexity of $f,g$ and $M_{\varepsilon}$ to see that,
for $x, y\in U$, $t\in [0,1]$,
\begin{eqnarray*}
& &
M_{\varepsilon}\left(f(tx+(1-t)y), g(tx+(1-t)y)\right)\leq \\
& &M_{\varepsilon}\left(tf(x)+(1-t)f(y), tg(x)+(1-t)g(y)\right)=\\
& &
M_{\varepsilon}\left(t(f(x), g(x))+(1-t)(f(y), g(y))\right)\leq \\
& & t M_{\varepsilon}(f(x), g(x))+(1-t)M_{\varepsilon}(f(y), g(y)),
\end{eqnarray*}
and, according to (viii) in the preceding lemma, the first
inequality is strict whenever $f$, $g$ are strictly convex and
$0<t<1$. To check (ix), it is sufficient to see that the function
$t\mapsto M_{\varepsilon}(f,g)(\gamma(t))$ has a strictly positive
second derivative at each $t$, where $\gamma(t)=x+tv$ with $v\neq 0$ (or, in the
Riemannian case, $\gamma$ is a nonconstant geodesic). So, by replacing $f, g$
with $f(\gamma(t))$ and $g(\gamma(t))$ we can assume that $f$ and
$g$ are defined on an interval $I\subseteq\R$ on which we have
$f''(t)>0, g''(t)>0$. But in this case we easily compute
\begin{eqnarray*}
& &\frac{d^2}{dt^2}M_{\varepsilon}(f(t), g(t))=\\
& &\frac{ \left(1+\theta'(f(t)-g(t))\right) f''(t) +
\left(1-\theta'(f(t)-g(t))\right) g''(t)}{2}
 + \frac{\theta''(f(t)-g(t)) \left( f'(t)-g'(t)\right)^{2}}{2}\geq\\
& &\geq \frac{1}{2}\min\{f''(t), g''(t)\}>0,
\end{eqnarray*}
because $|\theta'|\leq 1$ and $\theta''\geq 0$.
\end{proof}

\medskip

\begin{center}
{\bf Proof of Theorem \ref{main theorem}.}
\end{center}

Given a continuous convex function $f:U\to\R$ and $\varepsilon>0$,
we start defining $f_1=f$ and use the assumption that
$f_1-\varepsilon/2$ can be approximated from below by $C^k$ convex
functions, to find a $C^k$ convex function $h_1:U\to\R$ such that
$$ f_1 -\varepsilon\leq h_1 \textrm{
on } B_1, \textrm{ and }   h_1 \leq f_1-\frac{\varepsilon}{2} \textrm{
on } U.$$ We put $g_1=h_1$. Now define $f_2=f_1-\varepsilon$ and
find a convex function $h_2\in C^k(U)$ such that
$$f_2-\frac{\varepsilon}{2}\leq h_2 \textrm{ on } B_2, \textrm{ and } h_2 \leq f_2-\frac{\varepsilon}{4}
\textrm{ on } U.$$ Set
$$g_2=M_{\frac{\varepsilon}{10^2}}(g_1, h_2).$$ By the preceding
proposition we know that $g_2$ is a convex $C^k$ function
satisfying $$\max\{g_1, h_2\}\leq g_2\leq \max\{g_1,
h_2\}+\frac{\varepsilon}{10^{2}} \textrm{ on } U,$$ and
$$
g_{2}(x)=\max\{g_1(x), h_{2}(x)\} \textrm{ whenever } |h_{1}(x)-h_2(x)|\geq \frac{\varepsilon}{10^2}.
$$
\begin{claim}
We have
$$
g_2=g_1 \, \textrm{ on } \, B_1, \,\,\, \textrm{ and } \,\,\,
f-\varepsilon-\frac{\varepsilon}{2}\leq g_2\leq f-\frac{\varepsilon}{2}+\frac{\varepsilon}{10^2} \, \textrm{ on } \, B_2.
$$
\end{claim}
\noindent Indeed, if $x\in B_1$, $$g_1(x)\geq f_1(x)-\varepsilon=f_{2}(x)-\frac{\varepsilon}{4}+\frac{\varepsilon}{4}\geq h_2(x)+\frac{\varepsilon}{4}\geq h_2(x)+\frac{\varepsilon}{10^2},
$$
hence $g_2(x)=g_1(x)$, and in particular $f(x)-
\frac{\varepsilon}{2}\geq g_{2}(x)\geq f(x)-\varepsilon$. While,
if $x\in B_2\setminus B_1$ then
\begin{eqnarray*}
& &
f(x)-\varepsilon-\frac{\varepsilon}{2}\leq \max\{g_1(x), h_2(x)\}\leq g_{2}(x)\leq
\max\{g_1(x), h_2(x)\}+\frac{\varepsilon}{10^{2}}\leq\\
& & \max\{f(x)-\frac{\varepsilon}{2},
f(x)-\varepsilon-\frac{\varepsilon}{4}\}+
\frac{\varepsilon}{10^2}=f(x)-\frac{\varepsilon}{2}+\frac{\varepsilon}{10^2}.
\end{eqnarray*}
This proves the claim.

Next, define $f_3=f_2-\varepsilon/2=f-\varepsilon-\varepsilon/2$,
find a convex $C^k$ function $h_3$ on $U$ so that
$$
f_3-\frac{\varepsilon}{2^2}\leq h_3 \textrm{ on } B_3, \textrm{ and } h_3 \leq f_{3}-\frac{\varepsilon}{2^3} \, \textrm{ on } U,
$$
and set $$g_3=M_{\frac{\varepsilon}{10^3}}(g_2, h_3).
$$
\begin{claim} We have
$$
g_3=g_2 \, \textrm{ on } \, B_2, \,\,\, \textrm{ and } \,\,\,
f-\varepsilon-\frac{\varepsilon}{2}-\frac{\varepsilon}{2^{2}}\leq g_3
\leq f-\frac{\varepsilon}{2}+\frac{\varepsilon}{10^2}+\frac{\varepsilon}{10^3} \, \textrm{ on } \, B_3.
$$
\end{claim}
\noindent This is easily checked as before.

In this fashion we can inductively define a sequence of $C^k$ convex functions $g_n$ on $U$ such that
$$
g_n=g_{n-1} \, \textrm{ on } \, B_{n-1}, \,\,\, \textrm{ and } \,\,\,
$$
$$
f-\varepsilon-\frac{\varepsilon}{2}-\frac{\varepsilon}{2^{2}}-...-\frac{\varepsilon}{2^{n-1}}\leq g_n
\leq f-\frac{\varepsilon}{2}+\frac{\varepsilon}{10^2}+\frac{\varepsilon}{10^3}+...+\frac{\varepsilon}{10^n} \, \textrm{ on } \, B_n
$$
(at each step of the inductive process we define $f_{n}=f_{n-1}-\varepsilon/2^{n-2}=f-\varepsilon-...-\varepsilon/2^{n-2}$, we find $h_{n}$ convex and $C^k$ such that $f_{n}-\varepsilon/2^{n-1}\leq h_n$ on $B_n$ and $h_n\leq f_n-\varepsilon/2^n$ on $U$, and we put $g_n=M_{\varepsilon/10^n}(g_{n-1}, h_n)$).

Having constructed a sequence $g_n$ with such properties, we
finally define $$g(x)=\lim_{n\to\infty}g_n(x).$$ Since we have
$g_{n+k}=g_{n}$ on $B_n$ for all $k\geq 1$, it is clear that
$g=g_n$ on each $B_n$, which implies that $g$ is $C^k$ and convex on $U$
(or even strongly convex when the $g_n$ are strongly convex).
Besides, for every $x\in U=\bigcup_{n=1}^{\infty}B_n$ we have
$$
f(x)-2\varepsilon=f(x)-\sum_{n=1}^{\infty}\frac{\varepsilon}{2^{n-1}}\leq g(x)\leq f(x)-\frac{\varepsilon}{2}+\sum_{n=2}^{\infty}\frac{\varepsilon}{10^n},
$$
hence $f-2\varepsilon\leq g\leq f$. \, \, \, $\Box$

\begin{remark}\label{the method preserves strong convexity etc}
From the above proof and from Proposition \ref{properties of
M(f,g)} it is clear that this method of transferring convex
approximations on bounded sets to global convex approximations
preserves strict and strong convexity, local Lipschitzness,
minimizers and order, whenever the given approximations on bounded
sets have these properties.
\end{remark}

\section{Proofs of Corollaries \ref{uniform approximation of convex functions by smooth
convex functions in Rn}, \ref{uniform
approximation of convex functions on Cartan Hadamard manifolds}
and \ref{uniform approximation of convex functions on Banach
spaces}.}

We will deduce our corollaries by combining Theorem \ref{main theorem} with the known results on approximation of convex functions on bounded sets mentioned in the introduction, and with the following.

\begin{proposition}\label{uniform approximation of Lipschitz functions is enough}
Let $X$ be $\R^d$, or a Cartan-Hadamard manifold
(not necessarily finite-dimensional), or a Banach space, and let $U\subseteq X$ be open and convex.
Assume that $U$ has the property that every Lipschitz convex function on $U$ can be approximated
by $C^k$ convex (resp. strongly convex) functions, uniformly on $U$.

Then every convex function $f:U\to\R$
which is bounded on bounded subsets $B$ of $U$ with $\textrm{dist}(B, \partial U)>0$ can be approximated from below by $C^k$ convex
(resp. strongly convex) functions, uniformly on bounded subsets of $U$.
\end{proposition}
\begin{proof}
It is well known that a convex function
$f:U\to\R$ which is bounded on bounded subsets $B$ of $U$ with $\textrm{dist}(B, \partial U)>0$ is
also Lipschitz on each such subset $B$ of $X$.
So let $B\subset U$ be bounded, open and convex with $\textrm{dist}(B, \partial U)>0$, put
$L=\textrm{Lip}(f_{|_B})$, and define
$$
g(x)=\inf\{f(y)+L\, d(x,y) :y\in U\},
$$
where $d(x,y)=\|x-y\|$ in the case when $X$ is $\R^d$ or a Banach space, and $d$ is the Riemannian distance in $X$ when $X$ is a Cartan-Hadamard manifold.
\begin{claim}
The function $g$ has the following properties:
\begin{enumerate}
\item $g$ is convex on $X$.
\item $g$ is $L$-Lipschitz on $X$.
\item $g=f$ on $B$.
\item $g\leq f$ on $U$.
\end{enumerate}
\end{claim}
These are well known facts in the vector space case, but perhaps
not so in the Riemannian setting, so let us say a few words about
the proof. Property (iv) is obvious. To see that the reverse
inequality holds on $B$, take $x\in B$ and a subdifferential
$\zeta\in D^{-}f(x)$ (we refer to \cite{AFLM}, \cite{AF2} for the
definitions and some properties of the Fr{\'e}chet subdifferential and
inf convolution on Riemannian manifolds). We have $\|\zeta\|_x\leq
L$ because $f$ is $L$-Lipschitz on $B$. Since $\exp_{x}:TX_{x}\to
X$ is a diffeomorphism, for every $y\in X$ there exists $v_y\in
TX_x$ such that $\exp_{x}(v_y)=y$. And, because $t\mapsto f(\exp_x(tv_y))$ is
convex, we have $f(\exp_{x}(tv_y))-f(x)\geq \langle
\zeta, tv_y\rangle_x$ for every $t$, and in particular, taking
$t=1$, we get $f(y)-f(x)\geq \langle \zeta, tv_y\rangle_x\geq
-\|\zeta\|_x \|v_y\|_x\geq -L d(x,y)$. Hence $f(y)+ L d(x,y)\geq
f(x)$ for all $y\in X$, and taking the inf we get $g(x)\geq f(x)$.
Therefore $g=f$ on $B$. Showing (ii) is easy (as a matter of fact
this is true in every metric space). Finally, to see that $g$ is
convex on $X$, one does have to use that $X$ is a Cartan-Hadamard
manifold. We
note that in a Cartan-Hadamard manifold $X$ the distance function
$d:X\times X\to [0, \infty)$ is globally convex (see for instance \cite[V.4.3]{Sakai} and
\cite[Corollary 4.2]{AF2}), and that if $X\times U\ni(x,y)\mapsto
F(x,y)$ is convex then $x\mapsto \inf_{y\in U}F(x,y)$ is also
convex on $X$ (see \cite[Lemma 3.1]{AF2}). Since $(x,y)\mapsto f(y)+L
d(x,y)$ is convex on $X\times U$, this shows (i).

\medskip

Now, for a given $\varepsilon>0$, by assumption there exists a $C^k$ convex
(resp. strongly convex) function $\varphi:U\to\R$ so that $g-\varepsilon\leq\varphi\leq g$ on $U$.
Since $g\leq f$ on $U$, and $g=f$ on $B$, this implies that $\varphi\leq f$ on $U$,
and $f-\varepsilon\leq\varphi$ on $B$.
\end{proof}

Let $f:\R^d\to\R$ be continuous. As we recalled in the
introduction, if $\delta:\R^d\to [0, \infty)$ is a $C^\infty$
function such that $\delta(x)=0$ whenever $\|x\|\geq 1$, and
$\int_{\R^d}\delta=1$, then the functions
$f_{\varepsilon}(x)=\int_{\R^d}f(x-y)\delta_{\varepsilon}(y)dy$
(where
$\delta_{\varepsilon}(x)=\varepsilon^{-d}\delta(x/\varepsilon)$)
are $C^\infty$ and converge to $f(x)$ uniformly on every compact
set, as $\varepsilon\searrow 0$. Moreover, as is well known and
easily checked:
\begin{enumerate}
\item If $f$ is uniformly continuous then $f_{\varepsilon}$ converges to $f$
uniformly on $\R^d$.
\item If $f$ is convex (resp. strictly, or strongly convex), so is $f_{\varepsilon}$.
\item If $f$ is Lipschitz, so is $f_{\varepsilon}$, and $\textrm{Lip}(f_{\varepsilon})=\textrm{Lip}(f)$.
\item If $f$ is locally Lipschitz, $\textrm{Lip}(f_{{\varepsilon}_{|_B}})=\textrm{Lip}(f_{|_{(1+\varepsilon)B}})$ for every ball $B$.
\item If $f\leq g$ then $f_{\varepsilon}\leq g_{\varepsilon}$.
\item If $f$ is $C^1$ then $Df_{\varepsilon}$ converges to $Df$ uniformly on compact subsets of $\R^d$
\end{enumerate}
Therefore this method provides uniform approximation of Lipschitz
convex functions by $C^{\infty}$ convex functions, uniformly on
$\R^d$. By Proposition \ref{uniform approximation of Lipschitz
functions is enough} we then have that every (not necessarily
Lipschitz) convex function $f:\R^d\to\R$ can be approximated from
below by $C^\infty$ convex functions, uniformly on bounded sets.
And by Theorem \ref{main theorem} we get that every convex
function $f:\R^d\to\R$ can be approximated from below by
$C^\infty$ convex functions, uniformly on $\R^d$. Moreover, it is
clear that strict (or strong) convexity, local Lipschitzness, and
order are preserved by the combination of these techniques.

The case when $X=U$ is an open convex subset of $\R^d$ can be treated in
a similar manner. We consider the open, bounded convex sets
$B_m=\{x\in U
: \textrm{dist}(x, \partial U)>1/m, \|x\| <m\}$, so we have
$\overline{B_m}\subset B_{m+1}$, $\textrm{dist}(B_m, \partial U)>0$ and $U=\bigcup_{m=1}^{\infty}B_m$.
By combining Theorem \ref{main theorem} and Proposition
\ref{uniform approximation of Lipschitz functions is enough}, it
suffices to show that every Lipschitz, convex function $f:U\to\R$
can be approximated by $C^{\infty}$ convex functions, uniformly on
$U$. This can be done as follows: set
$L=\textrm{Lip}(f)$ and consider
    $$
g(x)=\inf\{f(y)+L \|x-y\| \, : \, y\in U\}, \,\,\, x\in\R^d,
    $$
which is a Lipschitz, convex extension of $f$ to all of $\R^d$,
with $\textrm{Lip}(f)=\textrm{Lip}(g)$. By using the above
argument, $g$ can be approximated by $C^\infty$ convex functions,
uniformly on $\R^d$. In particular, $f=g_{|_U}$ can be
approximated by such functions, uniformly on $U$. This proves
Corollary \ref{uniform approximation of convex functions by smooth
convex functions in Rn}.

Let us see how one can deduce Corollaries \ref{uniform
approximation of convex functions on Cartan Hadamard manifolds}
and \ref{uniform approximation of convex functions on Banach
spaces}. As in the case of $\R^d$, the combination of Theorem
\ref{main theorem}, Proposition \ref{uniform approximation of
Lipschitz functions is enough} and Remark \ref{the method
preserves strong convexity etc} reduces the problem to showing
that every {\em Lipschitz} convex function $f:X\to\R$ (where $X$
stands for a Cartan-Hadamard manifold or a Banach space whose dual
is locally uniformly convex) can be approximated by $C^1$ convex
functions, uniformly on $X$. It is well known that this can be
done via the inf convolution of $f$ with squared distances: the
functions
    $$
f_{\lambda}(x)=\inf\{ f(y)+ \frac{1}{2\lambda} d(x,y)^2 \, : \,
y\in X\}
    $$
are $C^1$, convex, Lipschitz (with the same constant as $f$), have
the same minimizers as $f$, are strictly convex whenever $f$ is and $X$ is reflexive (because in this case the inf defining $f_{\lambda}$ is always attained),
and converge to $f$ as $\lambda\searrow 0$, uniformly on all of
$X$. See \cite{Stromberg} for a survey on the inf convolution
operation in Banach spaces, and \cite{AF2} for the Cartan-Hadamard
case.

\section{Real analytic convex approximations}

Let us now prove Theorem \ref{uniform approximation of convex by real analytic convex}.
As mentioned in the introduction, real analytic approximations of partitions of
unity cannot be employed to glue local approximations into a
uniform convex approximation of $f$ on all of $\R^d$.

A natural approach to this problem would be showing that every convex function can be approximated by
$C^2$ strongly convex functions, and then using Whitney's theorem
on $C^2$-fine approximation of functions by real analytic
functions to conclude.
However, not every convex function
$f:\R^d\to\R$ can be approximated by strongly convex functions
uniformly on $\R^d$. For instance, it is not possible to
approximate a linear function by strongly convex functions.

We will show that, given a convex
function $f:\R^d\to\R$, either we can reduce the problem of
approximating $f$ by real analytic convex functions to some $\R^k$
with $k<d$, or else its graph is supported by a maximum of
finitely many {\em $(d+1)$-dimensional corners} which besides
approximates $f$ on a given bounded set (and which in turn we will
manage to approximate by strongly convex functions).

\begin{definition}[Supporting corners]
We will say that a function $C:\R^d\to\R$ is a
$k$-dimensional {\em corner function} on $\R^d$ if it is of the
form
    $$
C(x)=\max\{\, \ell_1 +b_1, \, \ell_2 +b_2, \, ..., \, \ell_k +b_k
\, \},
    $$
where the $\ell_j:\R^d\to\R$ are linear functions such that the
functions $L_{j}:\R^{d+1}\to\R$ defined by $L_{j}(x,
x_{d+1})=x_{d+1}-\ell_j(x)$, $1\leq j\leq k$, are linearly
independent, and the $b_j\in\R$. We will also say that a convex
function $f:U\subseteq\R^d\to\R$ is supported by $C$ at a point $x\in U$
provided we have $C\leq f$ on $U$ and $C(x)=f(x)$.
\end{definition}

\begin{lemma}\label{strongly convex approximation of corners}
If $C$ is a $(d+1)$-dimensional corner function on $\R^d$ then
$C$ can be approximated by $C^{\infty}$ strongly convex functions,
uniformly on $\R^d$.
\end{lemma}
\begin{proof}
We will need to use the following variation of the smooth maximum of Lemma \ref{smooth maxima}: given $\varepsilon, r>0$,
let $\beta_{\varepsilon, r}=|\cdot|*H_{r}+\varepsilon/2$, where $H_{r}(x)=\frac{1}{(4\pi r)^{1/2}}\exp(-x^{2}/4 r)$. We have $\beta_{\varepsilon, r}''(t)=2e^{-t^2/4 r}/(4 r\pi)^{1/2}>0$, so $\beta_{\varepsilon, r}$ is strongly convex and $1$-Lipschitz,  and as $r\to 0$ we have $\beta_{\varepsilon, r}(t)\to |t|+\varepsilon/2$ uniformly on $t\in\R$, so we may find $r=r(\varepsilon)>0$ such that $|t|\leq \beta_{\varepsilon, r}(t)\leq |t|+ \varepsilon$ for all $t$. Put $\widetilde{\theta}_{\varepsilon}(t)=\beta_{\varepsilon, r(\varepsilon)}(t)$, and define $\widetilde{M}_{\varepsilon}:\R^2\to\R$ by
$$
\widetilde{M}_{\varepsilon}(x,y)=\frac{x+y+\widetilde{\theta}_{\varepsilon}(x-y)}{2}.
$$
It is clear that $\widetilde{M}_{\varepsilon}$ satisfies all the properties of Lemma \ref{smooth maxima} except for $(iii)$.

Now let us prove our lemma. Up to an affine change of variables in $\R^{d+1}$, the problem is
equivalent to showing that the function
    $$
f(x)=\max\{ 0, x_1, x_2, ..., x_{d}\}
    $$
can be uniformly approximated on $\R^n$ by $C^{\infty}$ strongly
convex functions. We will show that this is possible by induction
on $d$.

For $d=1$, the function $f(x)=\max\{x,0\}$ is Lipschitz,
so the convolutions $f_{\varepsilon}=f*H_{\varepsilon}$  are $C^\infty$, Lipschitz and converge
to $f$, uniformly on $\R$, as $\varepsilon\searrow 0$. Besides, as
one can easily compute,
    $$
f''_{\varepsilon}(x)=\frac{1}{(4\pi\varepsilon)^{1/2}}
e^{-\frac{x^2}{4\varepsilon}}>0,
    $$
so the $f_{\varepsilon}$ are strongly convex.

Now, suppose that the function $f(x_{1}, ..., x_{k})=\max\{0,
x_1,...,x_k\}$ can be uniformly approximated by $C^\infty$ smooth
strongly convex functions on $\R^k$. Then, for a given
$\varepsilon>0$ we can find $C^\infty$ strongly convex functions
$g:\R^k\to\R$ and $\alpha:\R\to\R$ such that
    $$
f(x)\leq g(x)\leq f(x)+\varepsilon \,\,\, \textrm{ for all }
\,\,\, x\in\R^k, \,\,\, \textrm{ and } \,\,\,
    $$
    $$
\max\{t,0\}\leq\alpha(t)\leq\max\{t,0\}+\varepsilon \,\,\,
\textrm{for all }\,\,\, t\in\R.
    $$
Given the function $$F(x_1..., x_k, x_{k+1})=\max\{0, x_1, ...,
x_{k+1}\}=\max\{x_{k+1}, \, f(x_1, ..., x_k)\},$$ let us define
$G:\R^{k+1}\to\R$ by
    $$
G(x_1, ..., x_{k+1})=\widetilde{M}_{\varepsilon}\left( g(x_1, ..., x_k),
\alpha(x_{k+1})\right).
    $$
We have $G\in
C^{\infty}(\R^{k+1})$, and $F(x)\leq G(x)\leq F(x)+2\varepsilon$
for all $x\in \R^{k+1}$, so in order to conclude the proof we only
have to see that $G$ is strongly convex. Given $x, v\in\R^{k+1}$ with $v\neq 0$,
it is enough to check that the function
    $$
h(t):=G(x+tv)=\widetilde{M}_{\varepsilon}(\beta(t), \gamma(t)),
    $$
where $\beta(t)=g(x_1 +tv_1, ..., x_k+tv_k)$ and
$\gamma(t)=\alpha(x_{k+1} + tv_{k+1})$, satisfies $h''(t)>0$.
If $v_{k+1}\neq 0$ and $(v_{1}, ..., v_{k})\neq 0$ then,
since $g$ is strongly convex on $\R^k$ and $\alpha$ is strongly
convex on $\R$, we have $\beta''(t)>0$ and $\gamma''(t)>0$, so exactly as in the proof of $(9)$
of Proposition \ref{properties of M(f,g)} we also get
$h''(t)>0$. On the other hand, if for instance we have $v_{k+1}=0$ then $\beta''(t)>0$ and $\gamma'(t)=\gamma''(t)=0$, so
\begin{eqnarray*}
\frac{d^2}{dt^2}\widetilde{M}_{\varepsilon}(\beta(t), \gamma(t))=\frac{ \left(1+\theta_{\varepsilon}'(\beta(t)-\gamma(t))\right) \beta''(t) + \theta_{\varepsilon}''(\beta(t)-\gamma(t)) \left( \beta'(t)-\gamma'(t)\right)^{2}}{2}>0,
\end{eqnarray*}
because $|\widetilde{\theta}_{\varepsilon}'|< 1$ and $\widetilde{\theta}_{\varepsilon}''>0$. Similarly one checks that $\frac{d^2}{dt^2}\widetilde{M}_{\varepsilon}(\beta(t), \gamma(t))>0$ in the case when $(v_1,...,v_k)=0\neq v_{k+1}$.
\end{proof}

\begin{lemma}\label{reduction to Rk with k less than n}
Let $U\subseteq\R^n$ be open and convex, $f:U\to\R$ be a $C^p$ convex function, and $x_0\in U$.
Assume that $f$ is not supported at $x_0$ by any
$(n+1)$-dimensional corner function. Then there exist $k<n$, a
linear projection $P:\R^n\to\R^k$, a $C^p$ convex function
$c:P(U)\subseteq\R^k\to\R$, and a linear function $\ell:\R^n\to\R$ such that
$f=c\circ P+\ell$.
\end{lemma}
\begin{proof}
If $f$ is affine the result is obvious. If $f$ is not affine then
there exists $y_0\in U$ with $f'(x_0)\neq f'(y_0)$. It is clear
that $L_1(x, x_{n+1})=x_{n+1}-f'(x_0)(x)$ and $L_2(x,
x_{n+1})=x_{n+1}-f'(y_0)(x)$ are two linearly independent linear
functions on $\R^{n+1}$, hence $f$ is supported at $x_0$ by the
two-dimensional corner $x\mapsto \max\{f(x_0)+f'(x_0)(x-x_0),
f(y_0)+f'(y_0)(x-y_0)\}$. Let us define $m$ as the greatest
integer number so that $f$ is supported at $x_0$ by an
$m$-dimensional corner. By assumption we have $2\leq m<n+1$.
Define $k=m-1$. There exist $\ell_1, ..., \ell_{k+1}\in
(\R^n)^{*}$ with $L_{j}(x, x_{n+1})=x_{n+1}-\ell_j(x)$, $j=1, ...,
k+1$, linearly independent in $(\R^{n+1})^{*}$, and $b_{1}, ...,
b_{k+1}\in\R$, so that $C=\max_{1\leq j\leq k+1}\{\ell_j +b_j\}$
supports $f$ at $x_0$.

Observe that the $\{L_{j}-L_1\}_{j=2}^{k+1}$ are linearly
independent in $(\R^{n+1})^{*}$, hence so are the
$\{\ell_{j}-\ell_1\}_{j=2}^{k+1}$ in $(\R^n)^{*}$, and therefore
$\bigcap_{j=2}^{k+1}\textrm{Ker}\, (\ell_{j}-\ell_1)$ has
dimension $n-k$. Then we can find linearly independent vectors
$w_1, ..., w_{n-k}$ such that $\bigcap_{j=2}^{k+1}\textrm{Ker}\,
(\ell_{j}-\ell_1)=\textrm{span}\{w_1, ..., w_{n-k}\}$.

Now, given any $y\in U$, if $\frac{d}{dt} (f-\ell_1)(y+t
w_q)|_{t=t_{0}}\neq 0$ for some $t_0$ then
$f'(y+t_0 w_q)-\ell_1$ is linearly independent with
$\{\ell_{j}-\ell_1\}_{j=2}^{k+1}$, which implies that $(x,
x_{n+1})\mapsto x_{n+1}-f'(y+t_0w_q)$ is linearly independent with
$L_1, ..., L_{k+1}$, and therefore the function
$$x\mapsto \max\{\ell_1(x)+ b_1, ..., \ell_{k+1}(x)+b_{k+1},
f'(y+t_0 w_q)(x-y-t_0 w_q)+f(y+t_0 w_q)\}$$ is a $(k+2)$-dimensional corner supporting $f$ at $x_0$, which
contradicts the choice of $m$. Therefore we must have
    $$
\frac{d}{dt}(f-\ell_1)(y+tw_q)=0 \,\,\,
\textrm{ for all } \, y\in U, t\in\R \, \textrm{ with } \, y+tw_q\in U,  \, q=1, ..., n-k.
    $$
This implies that
$$
(f-\ell_1)(y+\sum_{j=1}^{n-k}t_{j}w_j)=(f-\ell_1)(y)
$$ if $y\in U$ and $y+\sum_{j=1}^{n-k}t_{j}w_j\in U$.
Let $Q$ be the orthogonal projection of $\R^n$ onto the subspace
$E:=\textrm{span}\{w_1, ..., w_{n-k}\}^{\bot}$. For each $z\in Q(U)$ we may define
$$
\widetilde{c}(z)=(f-\ell_1)(z+\sum_{j=1}^{n-k}t_j w_j)
$$
if $z+\sum_{j=1}^{n-k}t_j w_j \in U$ for some $t_1, ..., t_{n-k}$. It is clear that $\widetilde{c}:Q(U)\to\R$ is well defined,
convex and $C^p$, and satisfies
$$
f-\ell_1=\widetilde{c}\circ Q.
$$
Then, by taking a linear
isomorphism $T:E\to\R^k$ and setting $P=TQ$, we have that
$f=c\circ P+\ell_1$, where $c=\widetilde{c}\circ T^{-1}$ is defined on $P(U)$.
\end{proof}

Now we can prove Theorem \ref{uniform approximation of convex by
real analytic convex}. We already know that a convex function
$f:U\subseteq\R\to\R$ can be uniformly approximated from below by
$C^1$ functions, so we may assume that $f\in C^{1}(U)$.
We will proceed by induction on $d$, the dimension of $\R^d$.

For $d=1$ the result can be proved as follows. Either
$f:U\to\R$ is affine (in which case we are done) or $f$
can be supported by a $2$-dimensional corner at every point $x\in
U$. In the latter case, let us consider a compact interval
$I\subset U$. Given $\varepsilon>0$, since $f$ is convex and
Lipschitz on $I$ we can find finitely many affine functions $h_1,
..., h_m:\R\to\R$ such that each $h_j$ supports $f-\varepsilon$ at
some point $x_j\in I$ and $f-2\varepsilon\leq \max\{h_1, ...,
h_m\}$ on $I$. By convexity we also have $\max\{h_1, ..., h_{m}\}\leq f-\varepsilon$
on all of $U$. For each $x_j$ we may find a $2$-dimensional corner
$C_j$ which supports $f-\varepsilon$ at $x_j$. Since $f$ is
differentiable and convex we have $h_j=C_j$ on a neighbourhood of
$x_j$ and, by convexity, also $h_j\leq C_j\leq f-\varepsilon$ and
$\max\{C_1, ..., C_m\}\leq f-\varepsilon$ on $U$. And we also have
$f-2\varepsilon\leq \max\{h_1,..., h_m\}\leq\max\{C_1, ...,
C_{m}\}\leq f-\varepsilon$ on $I$. Now apply Lemma \ref{strongly
convex approximation of corners} to find $C^{\infty}$ strongly
convex functions $g_{1}, ..., g_{m}:\R\to\R$ such that $C_j\leq
g_j\leq C_j +\varepsilon'$, where $\varepsilon':=\varepsilon/2m$,
and define $g:\R\to\R$ by
$$g=M_{\varepsilon'}(g_1, M_{\varepsilon'}(g_2,
M_{\varepsilon'}(g_3, ..., M_{\varepsilon'}(g_{m-1}, g_m))...))$$
(for instance, if $m=3$ then $g=M_{\varepsilon'}(g_1,
M_{\varepsilon'}(g_2, g_3))$). By Proposition \ref{properties of
M(f,g)} we have that $g\in C^{\infty}(\R)$ is strongly convex,
    $$
\max\{C_1, ..., C_m\}\leq g\leq \max\{C_1, ...,
C_m\}+m\varepsilon'\leq f-\frac{\varepsilon}{2} \,\,\, \textrm{ on
} \,\,\, U,
    $$
and
    $$
f-2\varepsilon\leq \max\{C_1, ..., C_m\}\leq g \,\,\, \textrm{ on
} \,\,\, I.
    $$
Therefore $f:U\subseteq\R\to\R$ can be approximated from below by $C^\infty$
strongly convex functions, uniformly on compact subintervals of
$U$. By Theorem \ref{main theorem} and Remark \ref{the method
preserves strong convexity etc} we conclude that, given
$\varepsilon>0$ we may find a $C^\infty$ strongly convex function
$h$ such that $f-2\varepsilon\leq h\leq f-\varepsilon$ on $U$.

Finally, set $\eta(x)=\frac{1}{2}\min\{h''(x), \varepsilon\}$ for
every $x\in U$. The function $\eta:U\to (0, \infty)$ is
continuous, so we can apply Whitney's theorem on $C^2$-fine
approximation of $C^2$ functions by real analytic functions to find a real analytic function $g:U\to\R$ such
that
    $$
\max\{|h-g|, |h'-g'|, |h''-g''|\}\leq \eta.
    $$
This implies that $f-3\varepsilon\leq g\leq f$ and
$g''\geq\frac{1}{2}h''>0$, so $g$ is strongly convex as well.

\medskip

Now assume the result is true in $\R, \R^2, ..., \R^{d}$, and let
us see that then it is also true in $\R^{d+1}$. If there is some
$x_0\in U$ such that $f:U\subseteq\R^{d+1}\to\R$ is not supported
at $x_0$ by any $(d+2)$-dimensional corner function then,
according to Lemma \ref{reduction to Rk with k less than n}, we
can find $k\leq d$, a linear projection $P:\R^{d+1}\to\R^k$, a
linear function $\ell:\R^{d+1}\to\R$, and a $C^\infty$ convex
function $c:P(U)\to\R$ such that $f=c\circ P+\ell$. By assumption
there exists a real analytic convex function $h:P(U)\subseteq\R^k\to\R$ so that
$c-\varepsilon\leq h\leq c$. Then the function $g=h\circ P+\ell$
is real analytic, convex (though never strongly convex), and
satisfies $f-\varepsilon\leq g\leq f$.

If there is no such $x_0$ then one can repeat exactly the same
argument as in the case $d=1$, just replacing $2$-dimensional
corners with $(d+2)$-dimensional corners, the interval $I$ with a compact
convex body $K\subset U$, and $\eta$ with
    $$
\eta(x)=\frac{1}{2}\min\{ \varepsilon, \, \min\{D^{2}h(x)(v)^2 \,
: v\in \R^{d+1}, \|v\|=1\}\},
    $$
in order to conclude that there exists a real analytic strongly
convex $g:U\to\R$ such that $f-\varepsilon\leq g\leq f$ on $U$.
\,\,\, $\Box$

Incidentally, the above argument also shows Proposition
\ref{characterization of functions that cannot be approximated by
strongly convex functions} in the case when $f$ is $C^1$. In the
general case of a nonsmooth convex function one just needs to take
two more facts into account. First, Lemma \ref{reduction to Rk
with k less than n} holds for nonsmooth convex functions (to see
this, use the fact that if the range of the subdifferential of a
convex function is contained in $\{0\}$ then the function is
constant, see for instance \cite[Chapter 1, Corollary
2.7]{Clarke}, and apply this to the function $(t_1,...,
t_{d-k})\mapsto (f-\ell_1)(y+\sum_{j=1}^{d-k}t_j w_j)$). Second, in the
above proof one can use Rademacher's theorem and uniform
continuity of $f$ to see that the $x_j$ can be assumed to be
points of differentiability of $f$.

\section{$C^0$-fine approximation of general convex functions is impossible: three counterexamples}

We start to discuss the possibility of approximating a convex function $f:\R^d\to\R$ by smooth convex functions in the $C^{0}$-fine topology. We will see that there is quite a big difference between the cases $d=1$ and $d\geq 2$.

In the case $n=1$ we will show that every convex function $f:\R\to\R$ can be approximated by convex real analytic functions in this topology.
However, this approximation cannot be performed from below:
\begin{example}
Let $f:\R\to\R$ be defined by $f(x)=|x|$. For every $C^1$ convex function $g:\R\to\R$ such that $g(0)\leq 0$ we have $$\liminf_{|x|\to\infty}|f(x)-g(x)|>0.$$
In particular, if $\varepsilon:\R\to (0, \infty)$ is continuous and satisfies $\lim_{|x|\to\infty}\varepsilon(x)=0$ then there is no $C^1$ convex function $g:\R\to\R$ such that $|x|-\varepsilon(x)\leq g(x)\leq |x|$.
\end{example}

In two or more dimensions the situation gets much worse: $C^{0}$-fine approximation of convex functions by $C^1$ convex functions is no longer possible in general.

\begin{example}
For $d\geq 2$, let $f:\R^{d}\to\R$ be defined by $f(x_1, ..., x_d)=|x_1|$, and let $\varepsilon:\R^d\to (0, \infty)$ be continuous with $\lim_{|x|\to\infty}\varepsilon(x)=0$. Then there is no $C^1$ convex function $g:\R^d\to\R$ such that $|f-g|\leq\varepsilon$.
\end{example}

Our last example shows that when $U\neq\R^d$, $d\geq 2$, it is possible to construct convex functions $f:U\to\R$ which cannot be approximated by smoother convex functions in the $C^0$-fine topology, and which are not of the form $f=c\circ P+\ell$ (where $P:\R^d\to\R^k$, $k<d$, $c:P(U)\to\R$ convex and $\ell$ linear).

\begin{example}\label{counterexample to properly convex
approximation} Let $\varphi$ be a $C^p$ strongly convex function on $\R$
which is not $C^{p+1}$ on any neighbourhood of $0$, and let
$\psi:\R\to\R$ be a $C^{\infty}$ function such that $\psi=0$ on
$[-\varepsilon(1+\varepsilon), \varepsilon(1+\varepsilon)]$, and
$\min\{\psi, \psi''\}>0$ on
$\R\setminus[-\varepsilon(1+\varepsilon),
\varepsilon(1+\varepsilon)]$. Let $U=(-1,1)\times
(-1,1)\subset\R^2$, $\varepsilon\in (0,1)$, and define $f:U\to\R$ by
$$
f(x,y)=\varphi(x)+\psi\left(x+\varepsilon y\right)
+\psi\left(x-\varepsilon y\right).
$$
Notice that $f$ is strongly convex outside the set
$C_{\varepsilon}=\{(x,y)\in U \, : \,
-\varepsilon(1+\varepsilon)\leq x+\varepsilon
y\leq\varepsilon(1+\varepsilon),\, -\varepsilon(1+\varepsilon)\leq
x-\varepsilon y\leq\varepsilon(1+\varepsilon)\}$, and the measure of
$C_{\varepsilon}$ is less than $2\varepsilon(1+\varepsilon)$. It is
not difficult to see that if $\varepsilon:\R^2\to [0, \infty)$ is a
$C^1$ function with $\varepsilon^{-1}(0)=\R^2\setminus U$ then there
is no convex function $g\in C^{p+1}(U)$ such that $|f-g|\leq
\varepsilon$ on $U$.
\end{example}

\section{$C^{0}$-fine approximation of properly convex functions. A gluing technique for proper functions.}

We start proving Theorem 3. We may write $f=\ell +c$, where $\ell$
is linear and $c$ is convex and proper. Since addition of linear
functions preserves convexity, smoothness, and the kind of
approximation we are dealing with, in order to prove our result we
may assume that $\ell=0$, and in particular that $f:U\to [a, b)$ is
proper and attains a minimum at some point $x_0\in U$ with
$f(x_0)=a$.

For every $n\in\N$ let us define
$$
B_{n}=f^{-1}[a, \beta_n),
$$
where $(\beta_n)$ is a strictly increasing sequence of real numbers
converging to $b$. Each $\overline{B_n}=f^{-1}[a, \beta_n]$ is a compact convex body with interior $B_n$, and we have
$$
U=\bigcup_{n=1}^{\infty}B_n, \textrm{ and } \overline{B_n}\subset
B_{n+1} \textrm{ for all } n. \eqno(1)
$$
We also have
$$
\widetilde{\alpha}:=\inf_{U\setminus B_1}f-f(x_0)=\beta_1-a>0. \eqno(2)
$$
For each $v\in\R^d$ with $\|v\|=1$ let us consider the function
$\psi(t)=\psi_{x_0, v}(t)=f(x+tv)$. There are unique numbers
$\tau_{n}^{\pm}$ such that
$\tau_{n+1}^{-}<\tau_{n}^{-}<...<\tau_{1}^{-}<0<
\tau_{1}^{+}<...<\tau_{n}^{+}<\tau_{n+1}^{+}$ and
$x_0+\tau_{n}^{\pm}v \in\partial B_{n}$ for all $n$. By convexity of
$\psi$, for every $\eta_{n}^{\pm}\in\partial\psi (\tau_{n}^{\pm})$
we have $\eta_{n+1}^{-}\leq \eta_{n}^{-}\leq...\leq\eta_{1}^{-}\leq
0\leq \eta_{1}^{+}\leq ...\leq\eta_{n}^{+}\leq\eta_{n+1}^{+}$. Then,
for every $\zeta_{n}^{\pm}\in \partial f(x_0+\tau_{n}^{\pm}v)$ we
have that
$\eta_{n}=\zeta_{n}^{\pm}(v)\in\partial\psi(\tau_{n}^{\pm})$ and
therefore
$$
\|\zeta_{n}^{+}\|\geq\zeta_{n}^{+}(v)\geq\eta_{1}^{+}(v)\geq
\frac{\psi(\tau_1^{+})-\psi(0)}{\tau_{1}^{+}}\geq
\frac{\widetilde{\alpha}}{\textrm{diam}(B_1)}:=\alpha>0. \eqno(3)
$$
Since $v$ is an arbitrary unit vector, this shows in particular that
$$
\inf\{ \|\zeta\| \, : \, \zeta\in\partial f(y), y\in\partial B_n,
n\in\N\}\geq\alpha>0. \eqno(4)
$$
(A similar argument shows that if $v$ is a unit vector transversal to $\partial B_n$ at $y\in \partial B_n$ such that $y+tv\in B_n$ for $t>0$ sufficiently small, then the function $t\mapsto f(y+tv)$ is strictly decreasing on an interval $(-\delta, \delta)$, for some $\delta>0$ sufficiently small.)

Next, associated to each $B_n$ we define a function
$f_{n}:\R^d\to\R$ by
$$
f_n(x)=\inf\{ f(y)+L_{n+2}|x-y| \, : \, y\in\overline{B_{n+2}}\},
$$
where $L_{n+2}$ is the Lipschitz constant of $f_{|_{B_{n+2}}}$.
\begin{claim}
The $f_n$ are Lipschitz convex functions on $\R^d$ such that
$$
f_{n}\leq f_{n+1} \, \textrm{ on } \R^d,
$$
$$
f_{n}= f \textrm{ on } B_{n+2}.
$$
Moreover, $\lim_{|x|\to\infty}f_{n}(x)=\infty$, and $f_{n}$ can be
supported by a $(d+1)$-dimensional corner function at every point
$x\in\R^d$.
\end{claim}
\begin{proof}
It is a well known fact that $f_n$ is an $L_{n+2}$-Lipschitz convex extension of $f_{|_{B_{n+2}}}$ to all of $\R^d$, and it is easy to check that
$f_{n}\leq f_{n+1}$ for all $n$. Let us check that
$\lim_{|x|\to\infty}f_n(x)=\infty$. For every
$y\in\partial B_1$, there exists a unique unit vector $v=v_y$ such
that the ray $x_0 +tv, t>0$ intersects $\partial B_1$ at a unique
time $\tau_y$. Necessarily, $\tau_y\leq\textrm{diam}(B_1)$.
According to $(3)$ above, we have $\zeta_y(v)\geq\alpha$. Write
$x=x_{t,v}=x_0+tv$. By convexity of $f_n$ we have
\begin{eqnarray*}
& &f_n(x)=f_n(x_0+tv)\geq f_n(y)+(t-\tau_y)\zeta_y(v)=f(y)+(t-\tau_y)\zeta_y(v)\geq\\
& & f(y)+\alpha(t-\tau_y)\geq a
+\alpha(|x-x_0|-\textrm{diam}(B_1)),
\end{eqnarray*}
hence $f_{n}(x)\to\infty$ as $|x|\to\infty$. Finally,
according to Lemma \ref{reduction to Rk with k less than n}, if $f_n$ could not be
supported by a $(d+1)$-dimensional corner function at each
$x\in\R^d$ then we would have $f_{n}=c_n\circ P_n+\ell_n$
for some linear projection $P_n:\R^d\to\R^{k_n}$, $k_{n}<d$, $c_n:P_n(U)\to\R$
convex, and $\ell_n:\R^d\to\R$ linear. But this is impossible, since
for $y\in\textrm{Ker}P_n\setminus\{0\}$ we have
$c_n(P_n(ty))+\ell_n(ty)=c_n(0)+t\ell_n(y)$, which does not go to
$\infty$ as $|t|\to\infty$.
\end{proof}

Now, given a continuous function $\varepsilon:U\to (0, \infty)$,
define $$\varepsilon_n=\frac{1}{6}\min\{\varepsilon(x) :
x\in\overline{B_{n+1}}\}.$$ Associated to each $B_n$, and for every
number $r_n\in (0,1)$ let us also define functions
$\widetilde{f}_{n}=\widetilde{f}_{n, r_n}$ by
$$
\widetilde{f}_{n}(x)= (1-r_n)(f_n(x)-\beta_n) +\beta_n.
$$
\begin{claim}
The functions
$\widetilde{f}_{n}=\widetilde{f}_{n, r_n}$ are
convex and Lipschitz, and the $r_n$ can be chosen small enough so as
to have
$$
f_n < \widetilde{f}_{n} <f_n+\varepsilon_n \, \textrm{ on
} B_n,
$$
$$
f_n = \widetilde{f}_{n} \, \textrm{ on } \partial B_n,
$$
$$
\widetilde{f}_{n}<f_n \, \textrm{ on } \R^d\setminus\overline{B_n},
\, \textrm{ and }
$$
$$
f_n-\varepsilon_n<\widetilde{f}_{n}<f_n \, \textrm{ on }
B_{n+1}\setminus \overline{B_n}.
$$
Moreover, $\lim_{|x|\to\infty}\widetilde{f}_{n}(x)=\infty$, and
$\widetilde{f}_n$ can be supported by a $(d+1)$-dimensional corner
function at every point $x\in\R^d$.
\end{claim}
\begin{proof}
For $\varepsilon\in (0,1)$, denote $f_{n, \varepsilon}=(1-\varepsilon)(f_n(x)-\beta_n) +\beta_n$. It is clear that
$f_n<f_{n, \varepsilon}$ on $B_n$ and $f_{n, \varepsilon}<f_n$ on $\R^d\setminus\overline{B_n}$.
Since
$\lim_{\varepsilon\to 0^{+}}\widetilde{f}_{n,
\varepsilon}=f_n$ uniformly on compact subsets of $\R^d$,
we can find $\varepsilon=r_n\in (0,1)$ such that all the inequalities in the statement hold true.
On the other hand, by Claim 4 we get that
$\lim_{|x|\to\infty}\widetilde{f}_n(x)=\infty$, hence, by the same argument as in the proof of Claim 4,
$\widetilde{f}_n$ must also be supported by
$(d+1)$-dimensional corners at each point $x\in\R^d$.
\end{proof}

\begin{claim}
We can find numbers $\{r_n\}\subset (0,1)$ with $r_{n+1}<r_{n}$ for
all $n$, $r_n\searrow 0$ and as in the preceding claim, open convex
sets $A_n, C_n$, an open neighbourhood $\mathcal{N}_n$ of $\partial
A_n$, and numbers $s_n>0$ such that
$$
\overline{A_n}\subset B_n\subset\overline{B_{n}}\subset C_n\subset\overline{C_{n}}\subset A_{n+1},
$$
and the function $\widetilde{f}_{n}=\widetilde{f}_{n, r_n}$ satisfies
$$
\widetilde{f}_{n}+s_n\leq \min\{ f_n , \widetilde{f}_{n+1}\} \,
\textrm{ on } U\setminus C_n,
$$
$$
\widetilde{f}_n-\varepsilon_n\leq f_n\leq
\widetilde{f}_n+\varepsilon_n \, \textrm{ on } C_n, \,
\textrm{ and }
$$
$$
\widetilde{f}_n\geq \widetilde{f}_{n+1}+s_n \geq f+s_n \, \textrm{
on } A_n\cup \mathcal{N}_n.
$$
\end{claim}
\begin{proof}
This follows from Claim 5 and a standard compactness
argument.
\end{proof}

Now we are ready to construct a $C^\infty$ strongly convex function
$g:U\to\R$ such that $|f(x)-g(x)|\leq\varepsilon(x)$ for every $x\in
U$. We will do this by means of an inductive process. We start
considering the function $\widetilde{f}_1$. According to the proof
of Theorem \ref{uniform approximation of convex by real analytic convex}, because $\widetilde{f}_1$ can be supported by $(d+1)$-dimensional
corner functions at every point, we can find a strongly convex function
$\varphi_1\in C^{\infty}(U)$ such that
$\widetilde{f}_1-\varepsilon'_1\leq\varphi_1\leq \widetilde{f}_1$ on
$U$, where $\varepsilon'_1:=\frac{1}{2}\min\{\varepsilon_1, s_1\}$. Set
$g_1=\varphi_1$.
\begin{claim}\label{first estimation}
We have $|g_1-f|\leq 2\varepsilon_1$ on $C_1$.
\end{claim}
\begin{proof}
On $C_1$, on the one hand, $ g_1\leq \widetilde{f}_1\leq
f_1+\varepsilon_1\leq f+2\varepsilon_1$, and on the other hand, $
g_1\geq \widetilde{f}_1-\varepsilon'_1\geq
f_1-\varepsilon_1-\varepsilon'_1\geq f-2\varepsilon_1$.
\end{proof}

Next, consider $\widetilde{f}_2$, and set
$$
\delta_{2}:=\frac{s_1}{2}, \, \textrm{ and } \,
\varepsilon'_{2}:=\frac{1}{2}\min\{\varepsilon_2, s_2,
\varepsilon'_{1}\}.
$$
As before, we can find a strongly convex function $\varphi_2\in
C^{\infty}(U)$ such that $\widetilde{f}_{2}-\varepsilon'_{2}\leq
\varphi_2\leq \widetilde{f}_{2}$ on $U$. Define $g_2:U\to\R$ by
$$
g_2(x)=\left\{
         \begin{array}{ll}
           g_1(x), & \hbox{ if } x\in A_1 \\
           M_{\delta_2}(g_1(x), \varphi_2(x)), & \hbox{ if } x\in U\setminus A_1,
         \end{array}
       \right.
$$
where $M_{\delta_2}$ is the corresponding smooth maximum defined in Lemma \ref{smooth maxima}.

\begin{claim}
The function $g_2$ is well defined, strongly convex, $C^\infty$, and satisfies
$$
g_2=g_1 \, \textrm{ on } A_1,
$$
$$
|g_2-f|\leq 3\varepsilon_1 \, \textrm{ on } C_1,
$$
$$
g_{2}=\varphi_2 \, \textrm{ on } U\setminus C_1, \, \textrm{ and }
$$
$$
|g_2-f|\leq 2\varepsilon_2 \, \textrm{ on } C_2\setminus C_1.
$$
\end{claim}
\begin{proof}
Let $x\in\mathcal{N}_1$, then we have
\begin{eqnarray*}
g_1(x)=\varphi_1(x)\geq \widetilde{f}_{1}(x)-\varepsilon'_1 \geq
f(x)+s_1-s_1/2 =f_2(x)+\delta_2\geq\varphi_2(x)+\delta_2,
\end{eqnarray*}
hence $M_{\delta_2}(g_1(x), \varphi_2(x))=g_1(x)$. Using
Proposition \ref{properties of M(f,g)} this implies that $g_2$ is well defined,
convex and $C^\infty$. By definition $g_2=g_1$ on $A_1$. Let us see
that $g_{2}=\varphi_2 \, \textrm{ on } U\setminus C_1$. If $x\in
U\setminus C_1$,
\begin{eqnarray*}
g_1(x)=\varphi_1(x)\leq \widetilde{f}_1(x)\leq
\widetilde{f}_2(x)-s_1\leq
\varphi_2(x)+\varepsilon'_2-s_1\leq\varphi_2(x)-\delta_2
\end{eqnarray*}
hence $M_{\delta_2}(g_1(x), \varphi_2(x))=\varphi_2(x)$.

Let us now see that $|g_2-f|\leq 3\varepsilon_1$ on $C_1$. On the one hand we have, for every $x\in C_1$,
$$
g_2(x)\leq\max\{ g_1(x), \varphi_2(x)\}+\delta_2\leq \max\{f+2\varepsilon_1, f+\varepsilon_2\}+\delta_2\leq f+3\varepsilon_1,
$$
and on the other hand $g_2(x)\geq\max\{g_1(x), \varphi_2(x)\}\geq f(x)-2\varepsilon_1$.

Finally, on $C_2\setminus C_1$ we have $g_2=\varphi_2$ so, as in
Claim 7, we get $|g_2-f|\leq 2\varepsilon_2$ on $C_2\setminus C_1$.
\end{proof}
Now consider $\widetilde{f}_{3}$ and put
$$
\delta_{3}:=\frac{s_2}{2}, \, \textrm{ and } \,
\varepsilon'_{3}:= \frac{1}{2}\min\{\varepsilon_3, s_3,
\varepsilon'_{2}\},
$$
find a strongly convex function $\varphi_3\in C^{\infty}(U)$ such that $f_{3}-\varepsilon'_{3}\leq \varphi_3\leq f_3$ on $U$, and define
$$
g_3(x)=\left\{
         \begin{array}{ll}
           g_2(x), & \hbox{ if } x\in A_2 \\
           M_{\delta_3}(g_2(x), \varphi_3(x)), & \hbox{ if } x\in U\setminus
           A_2.
         \end{array}
       \right.
$$
As in the preceding claim, it is not difficult to check that $g_3$
is well defined, strongly convex, $C^\infty$, and satisfies
$$
g_3=g_2 \, \textrm{ on } A_2,
$$
$$
|g_3-f|\leq 3\varepsilon_2 \, \textrm{ on } C_2\setminus C_1,
$$
$$
g_{3}=\varphi_2 \, \textrm{ on } U\setminus C_2, \, \textrm{ and }
$$
$$
|g_3-f|\leq 2\varepsilon_3 \, \textrm{ on } C_3\setminus C_2.
$$
By continuing the inductive process in this manner one can construct a sequence of strongly convex functions $g_n\in C^{\infty}(U)$ such that
$$
g_{n+1}=g_n \, \textrm{ on } A_n,
$$
$$
|g_{n+1}-f|\leq 3\varepsilon_n \, \textrm{ on } C_n\setminus C_{n-1},
$$
$$
|g_{n+1}-f|\leq 2\varepsilon_{n+1} \, \textrm{ on } C_{n+1}\setminus C_n,
$$
and with $|g_1-f|\leq 2\varepsilon_1$ on $C_1$. This clearly implies that the function $g:U\to\R$ defined by
$$
g(x)=\lim_{n\to\infty}g_n(x)
$$
is $C^\infty$, strongly convex, and satisfies $|g(x)-f(x)|\leq\varepsilon(x)$ for all $x\in U$. Finally, in order to obtain a real analytic function $g$ with the same properties, one can apply Whitney's theorem on $C^2$-fine approximation of $C^2$ functions by real analytic functions, as in the last step of the proof of Theorem \ref{uniform approximation of convex by real analytic convex}.
$\Box$

\section{$C^{1}$-fine approximation of properly convex functions}

In order to prove Theorem \ref{a sufficent condition for C1 fine approximation property for convex functions on Rd} we will have to modify the proof of Theorem \ref{a sufficent condition for fine approximation property for convex functions on Rd} by carrying estimates on the derivatives, and take into account the following observation.
\begin{lemma}
Let $M_\varepsilon$ the smooth maximum of Lemma \ref{smooth maxima}, and let $V\subseteq\R^d$ be an open set.
If $\varphi, \psi\in C^{1}(V)$, then
$$
\| DM_{\varepsilon}(\varphi, \psi)-\frac{D\varphi +D\psi}{2}\|\leq\frac{1}{2}\|D\varphi-D\psi\|.
$$
\end{lemma}
\begin{proof}
Consider first the one-dimensional case when $\varphi, \psi: V\subseteq\R\to\R$. We have
$$
\frac{d}{dt}M_{\varepsilon}\left(\varphi(t), \psi(t)\right)=\frac{\varphi'(t)+\psi'(t)}{2}+\frac{1}{2}\theta_{\varepsilon}'\left(\varphi(t)-\psi(t)
\right) \left(\varphi'(t)-\psi'(t)\right),
$$
and $|\theta'_{\varepsilon}(s)|\leq 1$ for all $s$ because $\theta_{\varepsilon}$ is $1$-Lipschitz. Therefore
$$
\left|\frac{d}{dt}M_{\varepsilon}\left(\varphi(t), \psi(t)\right)-\frac{\varphi'(t)+\psi'(t)}{2}
\right| \leq \frac{1}{2}|\varphi'(t)-\psi'(t)|.
$$
The general case follows at once by considering, for every $x\in V$, $v\in\R^d$ with $\|v\|=1$, the functions $t\mapsto \varphi(x+tv)$ and $t\mapsto\psi(x+tv)$.
\end{proof}

Let us now explain the changes one has to make in the proof of Theorem \ref{a sufficent condition for fine approximation property for convex functions on Rd} in order to obtain Theorem \ref{a sufficent condition for C1 fine approximation property for convex functions on Rd}.
In this case we do not need to redefine the function $f$ outside $B_{n+2}$ (because we are not going to rely on the proof of Theorem \ref{uniform approximation of convex by real analytic convex}),
so we simply put $f_n=f$ and $\widetilde{f}_n=(1-r_n)(f_n-\beta_n)+\beta_n$. Notice that now we have $f_n, \widetilde{f}_n\in C^{1}(U)$ for every $n\in\N$.

We use the same preliminaries and Claims 4--6 (with obvious changes) as in the proof of Theorem \ref{a sufficent condition for fine approximation property for convex functions on Rd}, but in Claim 5 we add
$$
\|D\widetilde{f}_n-Df_n\|=\|D\widetilde{f}_n-Df\|\leq\varepsilon_n \, \textrm{ on } B_{n+2},
$$
which clearly holds provided $r_n>0$ is small enough.
Now we proceed with the inductive construction.
Consider the function $\widetilde{f}_1$. Notice that $\widetilde{f_1}$ is $C^1$ on $B_{3}\supset\overline{B_2}$. By using the convolutions
$(f_1-\varepsilon'_{1}/2)*\delta_t$, where
$\delta_{t}=t^{-d}\delta(x/t)$,
$\delta\geq 0$ being a $C^\infty$ function with bounded support and
$\int_{\R^d}\delta=1$, and taking $t>0$ sufficiently small, we can find a convex function
$\varphi_1\in C^{\infty}(U)$ of the form $\varphi_1=(\widetilde{f}_1 -\varepsilon'_{1}/2)*\delta_t$ such that
$\widetilde{f}_1-\varepsilon'_1\leq\varphi_1\leq \widetilde{f}_1$ and
$\|D\varphi_1-D\widetilde{f}_1\|\leq\varepsilon'_1$ on $\overline{B_{2}}$,
where $\varepsilon'_1:=\frac{1}{2}\min\{\varepsilon_1, s_1\}$. Set
$g_1=\varphi_1$.
\begin{claim}\label{first estimation}
We have $|g_1-f|\leq 2\varepsilon_1$ and $\|Dg_1-Df\|\leq 2\varepsilon_1$ on $C_1$.
\end{claim}
\begin{proof}
We only have to check the second inequality. On $C_1\subset B_2$ we have
$$
\|Dg_1-Df\|\leq\|D\varphi_1-D\widetilde{f}_1\|+\|D\widetilde{f}_1-Df\|\leq\varepsilon'_1+\varepsilon_1\leq 2\varepsilon_1.
$$
\end{proof}

Now consider $\widetilde{f}_2$, and set
$$
\delta_{2}:=\frac{s_1}{2}, \, \textrm{ and } \,
\varepsilon'_{2}:=\frac{1}{2}\min\{\varepsilon_2, s_2,
\varepsilon'_{1}\}.
$$
As before, we can find a convex function $\varphi_2\in
C^{\infty}(U)$ such that $\widetilde{f}_{2}-\varepsilon'_{2}\leq
\varphi_2\leq \widetilde{f}_{2}$ and $\|D\varphi_2-D\widetilde{f}_2\|\leq\varepsilon'_2$ on $\overline{B_3}$.
Define $g_2:U\to\R$ by
$$
g_2(x)=\left\{
         \begin{array}{ll}
           g_1(x), & \hbox{ if } x\in A_1 \\
           M_{\delta_2}(g_1(x), \varphi_2(x)), & \hbox{ if } x\in U\setminus A_1.
         \end{array}
       \right.
$$
\begin{claim}
The function $g_2$ is well defined, convex, $C^\infty$, and satisfies
$$
g_2=g_1 \, \textrm{ on } A_1,
$$
$$
|g_2-f|\leq 3\varepsilon_1 \, \textrm{ and } \, \|Dg_2-Df\|\leq 5\varepsilon_1 \, \textrm{ on } C_1,
$$
$$
g_{2}=\varphi_2 \, \textrm{ on } U\setminus C_1, \, \textrm{ and }
$$
$$
|g_2-f|\leq 2\varepsilon_2 \, \textrm{ and } \, \|Dg_2-Df\|\leq 2\varepsilon_2 \,  \textrm{ on } C_2\setminus C_1.
$$
\end{claim}
\begin{proof}
This time we only have to check the inequalities involving the derivatives. On $A_1$ we have $Dg_2=Dg_1$, so we have what we need by the preceding claim. On $C_{1}\setminus A_1$ we have
\begin{eqnarray*}
& & \|D g_{1}-D\varphi_2\|\leq \|Dg_1-D\widetilde{f}_1\|+\|D\widetilde{f}_1-Df\|+\|Df-D\widetilde{f}_2\|+\|D\widetilde{f}_2 -D\varphi_2\|\leq\\
& & \varepsilon'_1+\varepsilon_1+\varepsilon_2+\varepsilon'_2\leq 3\varepsilon_1,
\end{eqnarray*}
and therefore, using the preceding lemma,
\begin{eqnarray*}
& &
\|Dg_2-Df\|=
\|DM_{\delta_2}(g_1, \varphi_2)-Df\|\leq
\frac{1}{2}\|Dg_1-D\varphi_2\|+\frac{1}{2}\|Dg_1+D\varphi_2-2Df\|\leq \\
& & \frac{3}{2}\varepsilon_1+\frac{1}{2}\|Dg_1-Df\|+
\frac{1}{2}\|D\varphi_2-Df\|\leq \\
& &3\varepsilon_1+\frac{1}{2}\left( \|Dg_1-D\widetilde{f}_1\|+\|D\widetilde{f}_1-Df\|+
\|D\varphi_2-D\widetilde{f}_2\|+\|D\widetilde{f}_2-Df\| \right)\leq\\
& & 3\varepsilon_1 +\frac{1}{2}\left( \varepsilon'_1+\varepsilon_1+\varepsilon'_2+\varepsilon_2\right)\leq 5\varepsilon_1.
\end{eqnarray*}
Finally, on $C_2\setminus C_1$ we have $g_2=\varphi_2$, hence
$$
\|Dg_2-Df\|\leq \|D\varphi_2-D\widetilde{f}_2\|+\|D\widetilde{f}_2-Df\|\leq\varepsilon'_2+\varepsilon_2\leq 2\varepsilon_2.
$$
\end{proof}
Now consider $\widetilde{f}_{3}$ and put
$$
\delta_{3}:=\frac{s_2}{2}, \, \textrm{ and } \,
\varepsilon'_{3}:= \frac{1}{2}\min\{\varepsilon_3, s_3,
\varepsilon'_{2}\},
$$
find a convex function $\varphi_3\in C^{\infty}(U)$ such that $\widetilde{f}_{3}-\varepsilon'_{3}\leq \varphi_3\leq \widetilde{f}_3$ and
 $\|D\varphi_3 -D\widetilde{f}_3\|\leq\varepsilon'_3$ on $\overline{B_4}$, and define
$$
g_3(x)=\left\{
         \begin{array}{ll}
           g_2(x), & \hbox{ if } x\in A_2 \\
           M_{\delta_3}(g_2(x), \varphi_3(x)), & \hbox{ if } x\in U\setminus
           A_2.
         \end{array}
       \right.
$$
Again, it is not difficult to check that $g_3$
is well defined, convex, $C^\infty$, and satisfies
$$
g_3=g_2\, \textrm{ on } A_2,
$$
$$
|g_3-f|\leq 3\varepsilon_2 \, \textrm{ and } \, \|Dg_3-Df\|\leq 5\varepsilon_2 \, \textrm{ on } C_2\setminus C_1,
$$
$$
g_{3}=\varphi_3 \, \textrm{ on } U\setminus C_2, \, \textrm{ and }
$$
$$
|g_3-f|\leq 2\varepsilon_3 \, \textrm{ and } \, \|Dg_3-Df\|\leq 2\varepsilon_3 \,  \textrm{ on } C_3\setminus C_2.
$$
By continuing the inductive process in this manner one can construct a sequence of  convex functions $g_n\in C^{\infty}(U)$ such that
$$
g_{n+1}=g_n \, \textrm{ on } A_n,
$$
$$
|g_{n+1}-f|\leq 3\varepsilon_n \, \textrm{ and } \, \|Dg_{n+1}-Df\|\leq 5\varepsilon_n \, \textrm{ on } C_n\setminus C_{n-1},
$$
$$
|g_{n+1}-f|\leq 2\varepsilon_{n+1} \, \textrm{ and } \, \|Dg_{n+1}-Df\|\leq 2\varepsilon_{n+1} \,  \textrm{ on } C_{n+1}\setminus C_n,
$$
and with $|g_1-f|\leq 2\varepsilon_1\geq \|Dg_1-Df\|$ on $C_1$. This clearly implies that the function $g:U\to\R$ defined by
$$
g(x)=\lim_{n\to\infty}g_n(x)
$$
is $C^\infty$, convex, and satisfies $\max\{|g(x)-f(x)| \, , \, \|Dg(x)-Df(x)\| \}\leq\varepsilon(x)$ for all $x\in U$.
$\Box$

\begin{remark}
{\em The above proofs more generally show the following: if one has the ability to approximate $C^1$ properly convex functions by $C^{\infty}$ strongly convex functions, uniformly on compact sets, and in such a way that the derivatives of the approximations also approximate the derivatives of the given functions, uniformly on compact sets, then one can approximate $C^1$ properly convex functions by real analytic strongly convex functions, in the $C^1$ fine topology. We will investigate the general problem of uniformly approximating (not properly) convex functions and their derivatives in another paper. These proofs can also be easily adapted to get the following: let $M$ be a (noncompact) Riemannian manifold, and let $\mathcal{P}(M)$ be the class of convex functions $f:M\to\R$ such that $f(M)$ is an interval of the form $[a, b)$, with $-\infty<a<b\leq \infty$, and for every $\beta\in [a,b)$ the set $f^{-1}[a, \beta]$ is compact. If on $M$ one has the ability to approximate every function of $\mathcal{P}(M)$ by $C^p$ convex functions, uniformly on compact sets of $M$, then every function of $\mathcal{P}(M)$ can be approximated by $C^p$ convex functions in the $C^0$-fine topology. A similar statement holds for $C^1$-fine approximation.
By combining  this observation with \cite[Corollary 4.4]{AF2} we also deduce the following: if $M$ is a complete finite-dimensional Riemannian manifold with sectional curvature $K\leq 0$, then every function in $\mathcal{P}(M)$ can be approximated by $C^1$ convex functions in the $C^0$-fine topology. The condition that $f$ belong to $\mathcal{P}(M)$ cannot be removed in general, as we already know by considering the case when $M=\R^n$, or when $M$ is one of the manifolds constructed in \cite{Smith}.}
\end{remark}

\section{Proofs of Theorems \ref{C0 fine approximation of convex functions on R} and \ref{characterization of the fine
approximation property for convex functions on Rd}, and of Corollaries  \ref{prescribing continuous boundary values
and subdifferentials of real analytic convex functions} and \ref{prescribing differentials of smooth convex functions}.}

\medskip

\begin{center}
{\em Proof of Theorem \ref{C0 fine approximation of convex functions on R}.}
\end{center}

\noindent Let us first assume that $U=(a,b)$ with $-\infty<a<b<\infty$. If $f(a^{+}):=\lim_{t\to a^{+}}f(t)=\infty=\lim_{t\to b^{-}}f(t):=f(b^{-})$ then $f$ is proper, so by Theorem \ref{a sufficent condition for fine
approximation property for convex functions on Rd} we have what we need. If  these limits are both finite then we can write $f=c+\ell$, where $\ell(x)=\frac{f(b^{-})-f(a^{+})}{b-a}x$ is linear, and $c(a^{+})=c(b^{-})$, so either $c$ is constant, in which case we are done, or else $c$ is proper, and again we conclude by a direct application of Theorem \ref{a sufficent condition for fine approximation property for convex functions on Rd}.

Thus the only interesting case is when one of these limits is finite and the other one is infinite. Let us assume, for instance, that $\lim_{t\to a^{+}}f(t)<\infty=\lim_{t\to b^{-}}f(t)$. There exist $c, d\in (a,b)$ with $c<d$ and $f'(d)>f'(c)$. Define functions $f_{1}:(a,b)\to\R$ by
$$
f_1(x)=\left\{
  \begin{array}{ll}
    f(x), & \hbox{ if } a<x\leq d \\
    f(d)+f'(d)(x-d), & \hbox{ if } d\leq x<b,
  \end{array}
\right.
$$
and $f_2:(-\infty, b)\to\R$
$$
f_2(x)=\left\{
  \begin{array}{ll}
    f(d)+f'(d)(x-d), & \hbox{ if } x\leq d \\
    f(x), & \hbox{ if } d\leq x<b.
  \end{array}
\right.
$$
Notice that $f=\max\{f_1, f_2\}$ on $(a,b)$, and that $f_1$ and $f_2$ are properly convex on $(a,b)$ and $(-\infty, b)$, respectively. Moreover, there exist $\delta>0$ and $x_1, x_2\in (a,b)$ such that $x_1<x_2$, $f_1(x)\geq f_2(x)+\delta$ for all $x\in (a, x_1]$, and $f_2(x)>f_1(x)+\delta$ for all $x\in [x_2, b)$. Let $\varepsilon:(a,b)\to (0, \infty)$ be a continuous function. Put
$$
\varepsilon'=\frac{1}{2}\min\{ \delta, \, \min_{x\in [x_1, x_2]}\varepsilon(x) \, \},
$$
and
$$
\varepsilon_{1}(x)=\frac{1}{2}\min\{\varepsilon', \varepsilon(x)\}, \,\,\,
\varepsilon_{2}(x)=\left\{
                     \begin{array}{ll}
                       \varepsilon'/2, & \hbox{ if } x\in (-\infty, x_1] \\
                       \frac{1}{2}\min\{\varepsilon', \varepsilon(x)\}, & \hbox{ if } x\in [x_1, b).
                     \end{array}
                   \right.
$$
According to the proof of Theorem \ref{a sufficent condition for fine approximation property for convex functions on Rd}, we can find strongly convex functions $g_1, g_2\in C^{\infty}(a,b)$ such that $|f_1(x)-g_1(x)|\leq\varepsilon_1 (x)$ for all $x\in (a,b)$, and
$|f_2(x)-g_2(x)|\leq\varepsilon_2 (x)$ for all $x\in (-\infty,b)$. On $(a,b)$ define $g=M_{\varepsilon'}(g_1, g_2)$, which is a strongly convex $C^{\infty}$ function. We have $g=g_1$ on $(a,x_1]$, $g=g_2$ on $[x_2, b)$, and $|g(x)-f(x)|\leq\varepsilon(x)$ for every $x\in (a,b)$, as is easily checked. We can now conclude as in the last step of the proof of Theorem \ref{uniform approximation of convex by real analytic convex}.
The cases when $a=-\infty$ and (or) $b=+\infty$ can be treated in a similar manner.
$\Box$

\medskip

\begin{center}
{\em Proof of Theorem \ref{characterization of the fine
approximation property for convex functions on Rd}.}
\end{center}

\noindent It is easy to see that $(i)\iff(ii)\implies (iii)$.  We also have
$(i)\implies (iv)$ by Theorem \ref{a sufficent condition for fine
approximation property for convex functions on Rd}, and $(iv)\implies
(v)$ is trivial, so we only have to show $(v)\implies (iii)\implies
(ii)$. To see $(v)\implies (iii)$, suppose that $f=c\circ P+\ell$ and
that $f$ can be $C^0$-finely approximated by $C^{p+1}$ convex
functions. Let $\varepsilon:\R^{d}\to (0, \infty)$ be a continuous
function such that $\lim_{|x|\to\infty}\varepsilon(x)=0$. Find a
convex function $g\in C^{p+1}(\R^d)$ such that
$|f-g|\leq\varepsilon$. Then we will see that $f=g$, which
contradicts the assumption that $f\notin C^{p+1}(\R^d)$.

Suppose first that there exists $x\in \R^{d}$ such that $g(x)>f(x)$, and
take $v\in \textrm{Ker} P$, $v\neq 0$. Consider the convex function
$h(t)=g(x+tv)-t\ell (v)-f(x)=g(x+tv)-f(x+tv)$, which is defined on
$(-\infty, \infty)$. We have
$\lim_{|t|\to\infty}|f(x+tv)-g(x+tv)|=0$, hence also
$\lim_{|t|\to\infty}h(t)=0$. But $h(0)=g(x)-f(x)>0$, and this
contradicts the fact that $h$ is convex.

Therefore we must have $f-\varepsilon\leq g\leq f$ on $\R^{d}$. Now
assume that there exists $x\in \R^{d}$ such that $g(x)<f(x)$. For the
same function $h$ we now have
$h(0)=g(x)-f(x)<0=\lim_{|t|\to\infty}h(t)=0$. By the mean value
theorem there exists $t_0>0$ such that $h'(t_0)>0$, and by convexity
$h(t)\geq h(0)+h'(t_0)t$ for all $t>0$, which implies
$\lim_{t\to\infty}h(t)=\infty$, a contradiction. Therefore $f=g$ on
$\R^{d}$.

Finally, let us check $(iii)\implies (ii)$. By Lemma \ref{reduction to Rk with k less than n} there
exists a $(d+1)$-dimensional corner function $C$ which supports $f$
at $0$. And (for every $(d+1)$-dimensional corner function $C$ on
$\R^d$) it is easy to see that there exists a linear functional
$\ell:\R^{d}\to\R$ such that $C(x)-\ell(x)$ tends to $\infty$ as
$|x|\to\infty$. If we set $c=f-\ell$, we have $f(x)=c(x)+\ell(x)$,
with $c(x)=f(x)-\ell(x)\geq C(x)-\ell(x)\to\infty$ as
$|x|\to\infty$. $\Box$

\medskip

\begin{center}
{\em Proof of Corollary \ref{prescribing continuous boundary values
and subdifferentials of real analytic convex functions}.}
\end{center}

\noindent We may assume $\ell=0$.
Denote $V=\rm{int}(K)$. Take a $C^1$ function $\eta:\R^n\to
[0, \infty)$ such that $\eta^{-1}(0)=\R^n\setminus V$ and $\eta\leq
\varepsilon$ on $V$, use Theorem \ref{a sufficent condition for fine
approximation property for convex functions on Rd} to find a real
analytic strongly convex function $g:V\to\R$ such that
$|f-g|\leq\eta$ on $V$, and define $F:U\to\R$ by $F=f$ on
$U\setminus V$ and $F=g$ on $V$. Let us show that $F$ is convex near
$\partial V$. Take $x\in\partial V$ and $v\in\R^n$. We have to see
that $t\mapsto F(x+tv)$ is convex when $|t|$ is small. If $v$ is
tangent to $\partial V$, since $V$ is convex and $F=f$ on
$U\setminus V$, we have $F(x+tv)=f(x+tv)$, so this is obvious. If
$v$ is transversal to $\partial V$ at $x$, we can assume for
instance that there exists $\delta>0$ so that $x+tv\in V$ and
$x-tv\in U\setminus V$ for $t\in (0, \delta)$. Define
$\varphi_1(t)=f(x+tv)$ for $t\in (-\delta, \delta)$,
$\varphi_2(t)=g(x+tv)$ for $t\in [0, \delta)$, and
$\varphi:(-\delta, \delta)\to\R$ by $\varphi(t)=\varphi_1(t)$ if
$t<0$ and $\varphi(t)=\varphi_2(t)$ if $t>0$. We have to see that
$\varphi$ is convex, which amounts to checking that
$\varphi_{1}'(0^{-})\leq\varphi_{2}'(0^{+})$. And indeed, recalling
that $\eta=0$ on $U\setminus V$ and
$\frac{d}{dt}\eta(x+tv)_{|_{t=0}}=0$, and using convexity of
$\varphi_1$ on $(-\delta, \delta)$, we have
\begin{eqnarray*}
& &\lim_{t\to 0^{-}}\frac{\varphi_{1}(t)-\varphi_{1}(0)}{t}\leq\lim_{t\to 0^{+}}\frac{\varphi_{1}(t)-\varphi_{1}(0)}{t}\\
& &\leq \lim_{t\to 0^{+}}\frac{\varphi_{2}(t)-\varphi_{2}(0)+\eta(x+tv)}{t}=\lim_{t\to 0^{+}}\frac{\varphi_{2}(t)-\varphi_{2}(0)}{t}.
\end{eqnarray*}
To see that $\partial f(x)=\partial F(x)$, take $\zeta\in \partial
f(x)$ and assume that $\zeta\notin\partial F(x)$, then there is
$v\neq 0$ such that the line $t\mapsto F(x)+t\zeta(v)$ does not
support $F(x+tv)$ at $t=0$. As before we may assume that $v$ is
transversal to $\partial V$ at $x$ and also, up to replacing $v$
with $-v$, that $x+tv\in V$ and $x-tv\in U\setminus V$ for $t\in (0,
\delta)$. Let $\varphi_1, \varphi_2$ be defined as above. We have,
for small $s_1, s_2>0$,
\begin{eqnarray*}
\frac{F(x-s_1 v)-F(x)}{-s_1}\leq\zeta(v)\leq\lim_{t\to 0^{-}}\frac{\varphi_{1}(t)-\varphi_{1}(0)}{t}\leq
\lim_{t\to 0^{+}}\frac{\varphi_{2}(t)-\varphi_{2}(0)}{t} \leq \frac{F(x+s_2 v)-F(x)}{s_2},
\end{eqnarray*}
which contradicts the assumption that the line $t\mapsto
F(x)+t\zeta(v)$ does not support $F(x+tv)$ at $t=0$. Similarly one
sees that $\partial F(x)\subseteq \partial f(x)$.  Finally, in the
case when $f\in C^1(U)$, $\partial f(x)$ is a singleton for every
$x\in\partial V$, hence so is $\partial F(x)$, and therefore $F$ is
differentiable at every point of $\partial V$. Since a
differentiable convex function always has a continuous derivative,
it follows that $F\in C^{1}(U)$. $\Box$

\medskip

The proof of Corollary \ref{prescribing differentials of smooth convex functions} is easier, and we leave it to the reader's care.

\section{Appendix: Convex functions vs convex bodies}

In this appendix we recall a (somewhat unbalanced) basic relationship between convex functions and convex bodies, regarding approximation. Given a convex function $f:\R^d\to\R$, if we consider the epigraph $C$ of $f$, which is an unbounded convex body in $\R^{d+1}$, we can approximate $C$ by smooth convex bodies $D_k$ such that $\lim_{k\to\infty} D_k=C$ in the Hausdorff distance. Then it is easy to see (via the implicit function theorem) that the boundaries $\partial D_k$ are graphs of smooth convex functions $g_k:\R^d\to\R$ such that $\lim_{k\to\infty} g_k=f$ uniformly on compact subsets of $\R^d$. But when $f$ is not Lipschitz this convergence is not uniform on $\R^d$, as the following example shows.

\begin{example}
Consider the function $f:\R\to\R$, $f(x)=x^2$. The epigraph $C:=\{(x,y): y\geq x^2\}$ is an unbounded convex body, and the set $D:=\{(x,y) : \textrm{dist} \left( (x,y), C\right)\leq \varepsilon/2\}$ is a $C^1$ convex body such that $C\subset D\subset C+\varepsilon B$, where $B$ is the unit ball of $\R^2$. Hence $D$ approximates $C$ in the Hausdorff distance, and the boundary $\partial D$ is indeed the graph of a $C^1$ convex function $g:\R\to\R$. But the function $g$ does not approximate $f$ on $\R$, because $\lim_{|x|\to\infty} |f(x)-g(x)|=\infty$.
\end{example}

Therefore one cannot employ results on approximation of (unbounded) convex bodies to deduce results on global approximation of convex functions.
By contrast, one can use the well known results on global approximation of Lipschitz convex functions by real analytic convex functions to deduce
the following result (first proved by Minkowski in the case when $C$ is bounded):

\begin{theorem}
Let $C\subset\R^d$ be a (not necessarily bounded) convex body. For every $\varepsilon>0$ there exists a real analytic convex body $D$ such that
$$
C\subset D\subset C+\varepsilon B,
$$
where $B$ is the unit ball of $\R^d$.
\end{theorem}
\begin{proof}
Consider the $1$-Lipschitz, convex function $f:\R^n\to [0, \infty)$ defined by $f(x)=\textrm{dist}(x,C)$. Using integral convolution with the heat kernel one can produce a real analytic convex (and $1$-Lipschitz) function $g:\R^n\to\R$ such that $f-2\varepsilon/3\leq g\leq f-\varepsilon/3$ on $\R^n$. Define $D=g^{-1}(-\infty, 0]$. Since $g$ is convex and does not have any minimum on $\partial D=g^{-1}(0)$, we have $\nabla g(x)\neq 0$ for all $x\in \partial D$, hence $\partial D$ is a $1$-codimensional real analytic submanifold of $\R^n$. Because $f\geq g$, we have $C\subset D$. And if $x\notin C+\varepsilon B$ then $f(x)\geq \varepsilon$, hence $g(x)-\varepsilon/3\geq f(x)-\varepsilon\geq 0$, which implies $g(x)>0$, that is $x\notin D$.
\end{proof}

\bigskip



\begin{thebibliography}{99}

\bibitem{Alexandroff}

A.D. Alexandroff, {\em Almost everywhere existence of the second
differential of a convex function and some properties of convex
surfaces connected with it}, Leningrad State Univ. Annals (Uchenye
Zapiski) Math. Ser. 6, (1939). 3--35.

\bibitem{AF2}
D. Azagra, and J. Ferrera, {\em Inf-convolution and regularization
of convex functions on Riemannian manifolds of nonpositive
curvature}, Rev. Mat. Complut. 19 (2006), no. 2, 323--345.

\bibitem{AFLM}
D. Azagra, J. Ferrera, F. L{\'o}pez-Mesas, {\em Nonsmooth analysis and
Hamilton-Jacobi equations on Riemannian manifolds}, J. Funct.
Anal. 220 (2005), no. 2, 304--361.

\bibitem{Bangert}
V. Bangert, {\em Analytische Eigenschaften konvexer Funktionen auf
Riemannschen Mannigfaltigkeiten}, J. Reine Angew. Math. 307/308
(1979), 309--324.

\bibitem{Bangert2}
V. Bangert, {\em {\"U}ber die Approximation von lokal konvexen
Mengen}, Manuscripta Math. 25 (1978), no. 4, 397--420.

\bibitem{Clarke}
F. H. Clarke,  Yu. S. Ledyaev,  R. J. Stern, P. R. Wolenski, {\em
Nonsmooth analysis and control theory}, Graduate Texts in
Mathematics, 178. Springer-Verlag, New York, 1998.

\bibitem{CheegerGromoll}
J. Cheeger, and D. Gromoll, {\em On the structure of complete
manifolds of nonnegative curvature}, Ann. of Math. 96 (1972),
413--443.

\bibitem{DFH1}
R. Deville, V. Fonf, P. H{\'a}jek, {\em Analytic and $C^k$
approximations of norms in separable Banach spaces}, Studia Math.
120 (1996), no. 1, 61--74.

\bibitem{DFH2}
R. Deville, V. Fonf, P. H{\'a}jek, {\em Analytic and polyhedral
approximation of convex bodies in separable polyhedral Banach
spaces}, Israel J. Math. 105 (1998), 139--154.

\bibitem{Ghomi1}
M. Ghomi, {\em The problem of optimal smoothing for convex
functions.} Proc. Amer. Math. Soc. 130 (2002), no. 8, 2255--2259.

\bibitem{Ghomi2}
M. Ghomi, {\em Optimal smoothing for convex polytopes.} Bull. London
Math. Soc. 36 (2004), no. 4, 483--492

\bibitem{Greene1}
R. E. Greene, and K. Shiohama, {\em Convex functions on complete
noncompact manifolds: topological structure}, Invent. Math. 63
(1981), no. 1, 129--157.

\bibitem{Greene2}
R. E. Greene, and K. Shiohama, {\em Convex functions on complete
noncompact manifolds: differentiable structure}, Ann. Sci. {\'E}cole
Norm. Sup. (4) 14 (1981), no. 4, 357--367 (1982).

\bibitem{Greene5}
R. E. Greene, and H. Wu, {\em On the subharmonicity and
plurisubharmonicity of geodesically convex functions}, Indiana
Univ. Math. J. 22 (1972/73), 641--653.

\bibitem{Greene3}
R. E. Greene, and H. Wu, {\em $C\sp{\infty }$ convex functions and
manifolds of positive curvature}, Acta Math. 137 (1976), no. 3-4,
209--245.

\bibitem{Greene4}
R. E. Greene, and H. Wu, {\em $C\sp{\infty }$ approximations of
convex, subharmonic, and plurisubharmonic functions}, Ann. Sci.
{\'E}cole Norm. Sup. (4) 12 (1979), no. 1, 47--84.

\bibitem{GromollMeyer}
D. Gromoll, and W. Meyer, {\em On complete open manifolds of
positive curvature}, Ann. of Math. 90 (1969) 75--90.

\bibitem{Gruber}
P. Gruber, {\em  Aspects of approximation of convex bodies. Handbook of convex geometry}, Vol. A, B, 319--345, North-Holland, Amsterdam, 1993.

\bibitem{Rademacher}
H. Rademacher {\em {\"U}ber partielle und totale Differenzierbarkeit
I.}, Math. Ann. 89 (1919), 340--359.

\bibitem{Rockafellar}
R.T. Rockafellar, {\em Convex analysis}. Princeton Mathematical Series, No. 28. Princeton University Press, Princeton, N.J., 1970

\bibitem{Sakai}
T. Sakai, {\em Riemannian Geometry}, Translations of Mathematical Monographs, 149.
American Mathematical Society, Providence, RI, 1996.

\bibitem{Schneider}
R. Schneider, {\em Convex bodies: the Brunn-Minkowski theory}. Encyclopedia of Mathematics and its Applications, 44. Cambridge University Press, Cambridge, 1993.

\bibitem{Smith}
P. A. N. Smith, {\em Counterexamples to smoothing convex
functions}, Canad. Math. Bull. 29 (1986), no. 3, 308--313.

\bibitem{Stromberg}
T. Str{\"o}mberg, {\em The operation of infimal convolution},
Dissertationes Math. (Rozprawy Mat.) 352 (1996)

\bibitem{Whitney}
H. Whitney, Analytic extensions of differential functions in
closed sets, Trans. Amer. Math. Soc. 36 (1934), 63--89.


\end{thebibliography}
\end{document}